\documentclass[10pt,reqno]{amsart}

\usepackage[latin1]{inputenc}
\usepackage[french,british]{babel}
\usepackage[T1]{fontenc}
\usepackage{setspace}
\usepackage{indentfirst} 
\usepackage{fancyhdr} 
\usepackage{nicefrac}
\usepackage{textcomp}
\usepackage{lastpage}
\usepackage{amssymb,amsfonts,amsmath,amsthm}
\usepackage{latexsym}
\usepackage{hyperref}
\usepackage{url}
\usepackage{verbatim}
\usepackage[all]{xy}
\usepackage{mathrsfs}
\usepackage{stmaryrd}
\usepackage{oldgerm}
\usepackage{bm}
\usepackage{graphics}
\usepackage{fancybox}
\usepackage{xcolor}
\usepackage{graphicx}
\usepackage{ulem}
\input xy
\xyoption{all}

\theoremstyle{plain}
\newtheorem{theorem}{Theorem}[section]
\newtheorem*{theorem*}{Theorem}
\newtheorem{lemma}[theorem]{Lemma}
\newtheorem*{lemma*}{Lemma}
\newtheorem{corollary}[theorem]{Corollary}
\newtheorem{proposition}[theorem]{Proposition}
\newtheorem*{proposition*}{Proposition}

\newtheorem*{conjecture*}{Conjecture}

\theoremstyle{remark}
\newtheorem{remark}[theorem]{Remark}

\theoremstyle{definition}
\newtheorem{definition}[theorem]{Definition}

\numberwithin{equation}{section}
\renewcommand\labelenumi{(\roman{enumi})}
\renewcommand\theenumi\labelenumi

\DeclareMathOperator{\QQ}{\mathbf{Q}}

\DeclareMathOperator{\Gal}{Gal}

\DeclareMathOperator{\pr}{pr}
\DeclareMathOperator{\Spec}{Spec}
\DeclareMathOperator{\Frac}{Frac}

\newcommand{\F}{\mathscr{F}}
\newcommand{\Sch}{\mathsf{Sch}}
\newcommand{\Reg}{\mathsf{Reg}}

\newcommand{\Dvr}{\mathsf{Dvr}}
\newcommand{\Shdvr}{\mathsf{Shdvr}}

\newcommand{\Zar}{\textnormal{Zar}}
\newcommand{\Nis}{\textnormal{Nis}}
\newcommand{\et}{\textnormal{\'{e}t}}
\newcommand{\cdp}{\textnormal{cdp}}
\newcommand{\qfh}{\textnormal{qfh}}
\newcommand{\h}{\textnormal{h}}
\newcommand{\eh}{\textnormal{eh}}
\newcommand{\cdh}{\textnormal{cdh}}
\newcommand{\rh}{\textnormal{rh}}
\newcommand{\fin}{\textnormal{fin}}
\newcommand{\prop}{\textnormal{prop}}

\newcommand{\dvr}{\textnormal{dvr}}
\newcommand{\shdvr}{\textnormal{shdvr}}
\newcommand{\sh}{\textnormal{sh}}
\newcommand{\ess}{\mathsf{ess}}

\DeclareMathOperator{\Frob}{F}
\DeclareMathOperator{\Ver}{V}
\DeclareMathOperator{\tor}{tor}

\newcommand{\U}{\mathfrak{U}}
\newcommand{\OO}{\mathscr{O}}
\newcommand{\sbt}{\,\begin{picture}(-1,1)(-1,-2)\circle*{2}\end{picture}\ }

\title{ Witt differentials in the h-topology}

\thanks{2010\,\emph{Mathematics Subject Classification}.
 14F40 (primary); 13F35 (secondary)}

\author{Veronika Ertl}
\address{ Universit\"at Regensburg, Fakult\"at f\"ur Mathematik Universit\"atsstrasse 31, 93053 Regensburg, Germany}
\email{{veronika.ertl@mathematik.uni-regensburg.de}}

\author{Lance Edward Miller}
\address{Department of Mathematical Sciences, 309 SCEN, University of Arkansas,
Fayetteville, AR, 72701}
\email{{lem016@uark.edu}}

\begin{document}

\maketitle

\begin{abstract} 
For sheaves of differential forms of the de~Rham--Witt complex for regular varieties over a perfect field of positive characteristic $p$ we prove  unconditional descent in cohomological dimension $0$ with respect to Voevodsky's $\h$-topology. Under resolution of singularities we obtain full cohomological descent. Our approach follows recent work of Huber--J\"order and Huber--Kebekus--Kelly on sheaves of differential forms.
\end{abstract}

\section{Introduction}

The idea to use differential forms to obtain numerical invariants of algebraic varieties dates back to Picard and Lefschetz. Meanwhile sheaves of differential forms have become an important tool in the study of local and global properties of algebraic varieties and schemes. However it is well-known that they are not well behaved in the singular case even in characteristic $0$. To get around the problems which occur here, several competing generalizations of differential forms and of the de~Rham complex have been proposed. Examples include K\"ahler differentials in the case of a variety which is embeddable into a smooth variety, reflexive differentials for normal varieties,  logarithmic differentials in the case a divisor is considered,  Deligne's de~Rham complex using proper hypercovers, or the Du~Bois complex. 

In \cite{HuberJoerder2014} Huber and J\"order introduce another player using Voevodsky's $\h$-topology \cite{Voevodsky1996}. It turns out that in characteristic $0$ the $\h$-sheafification of the sheaves $\Omega^n$, $n\geqslant 0$, of differential forms provides a conceptual extension to singular varieties which recovers  the above variants in many situations. This approach is built on the intuition that all varieties over a field of characteristic $0$ are $\h$-locally smooth.

In  characteristic $p>0$ the $\h$-topology encounters a number of challenges. For example, one quickly notes that the basic K\"ahler differentials become zero under $\h$-sheafification. This is due to the fact that the Frobenius map is an $\h$-cover.  However, it is possible to circumvent these difficulties to some extent with  more subtle  sheaf theoretic methods \cite{HuberKebekusKelly2016}.

In positive characteristic, one is often led to consider $p$-adic cohomolgy theories instead. If a variety over a perfect field of  characteristic $p$ lifts to characteristic $0$, the first choice is to consider  the differential forms of this lift as described above. If a variety lifts only locally however, one often turns towards crystalline cohomology or one of its variants. In \cite{Illusie1979} Illusie introduced a complex of \'etale sheaves, called the de~Rham--Witt complex, which computes crystalline cohomology on smooth schemes.  

The hope is to extend the program described in \cite{HuberJoerder2014} and \cite{HuberKebekusKelly2016} to the de~Rham--Witt complex in order to obtain an equally conceptual approach to the study of singular varieties in characteristic $p$. In this article, we start by studying the $\h$-sheafification of the rational sheaf of Witt differential forms $W\Omega^n_{\QQ}$ of degree $n\geqslant 0$. This  extends  the  Witt differential forms to singular varieties analogous to Huber--J\"order's method described above. 

To lay the base for this, we are   especially interested in descent results for rational Witt differentials. An optimal result would be a cohomological descent statement analogous to \cite[Cor.~6.5]{HuberJoerder2014}. This appears to be well-known  to experts if one assumes resolution of singularities in positive characteristic. We give a detailed proof for this in Section~\ref{sec:res}. Without the assumption of resolution of singularities such a statement is only known for Serre's Witt vector cohomology \cite{BhattScholze2015}. For Witt differential forms of arbitrary degree full cohomological descent without resolutions of singularities remains an unsurmountable challenge. 

However, using  techniques from \cite{HuberKebekusKelly2016} we were able to obtain the following descent result for Witt differentials of any degree without resolution of singularities.

\begin{theorem*}[Thm.~\ref{dRWSmHDescent}]
Let $k$ be a perfect field of positive characteristic $p$. For any regular scheme $X$ over $k$, the change of topology map induces quasi-isomorphisms 
	$$
	W\Omega^n_{\QQ}(X)\cong W\Omega^n_{\QQ,\h}(X)
	$$ 
for all $n\geqslant 0$.
\end{theorem*}
The proof relies on a construction introduced in \cite{HuberKebekusKelly2016} which allows one to extend  a presheaf $\F$ on regular schemes to a presheaf $\F_{\dvr}$ on arbitrary schemes. We study it closer in Section~\ref{sec:dvr}.

An important insight by Huber, Kebekus and Kelly is that this construction extends sheaves for topologies coarser than the \'etale topology. More precisely, if $\F$ is for example an \'etale sheaf on regular schemes, $\F_{\dvr}$ is an $\eh$-sheaf on $\Sch(k)$. In Corollary~\ref{cor:qfhSheaf}  we show a similar result for the $\qfh$-topology. Concretely, this means, that in order to show that $\F_{\dvr}$ is a $\qfh$-sheaf, it suffices to show that $\F$ is a $\qfh$-sheaf on a certain smaller category. Following a modification of the approach in \cite{HuberKebekusKelly2016} suggested to us by Kelly, we obtain the following. 

\begin{proposition*}[Prop.~\ref{prop:qfhDescent}]
For a perfect field $k$ of characteristic $p > 0$, the extension $W\Omega^n_{\QQ,\dvr}$ is a $\qfh$-sheaf on $\Sch(k)$.
\end{proposition*}
The proof of this explicitly uses properties of Witt differentials, notably the existence of a Frobenius and a Verschiebung map. As $\eh$- and $\qfh$-topology generate the $\h$-topology, this result implies that $W\Omega^n_{\QQ,\dvr}$ is an $\h$-sheaf on $\Sch(k)$. 

Another ingredient for our main theorem  is Corollary~\ref{Cor:Torsion} which states that if $\F_{\dvr}$ is an $\h$-sheaf, then it has no topological torsion. In particular this is true for $W\Omega^n_{\QQ,\dvr}$. It is applied in the last section to show that the natural map $W\Omega_{\QQ,\h}^n(X) \to W\Omega_{\QQ,\dvr}^n(X)$ is injective. Our main result then follows by a diagram chase. 

To finish, we draw several expected consequences from the $\h$-descent. This includes analogues for the rational  Witt differentials of results from \cite{HuberJoerder2014} on K\"ahler differentials (see Theorem~\ref{prop:BasicProperties}). 

\subsection*{Acknowledgements}
 We are indebted to Annette Huber, without her this paper would not exist. Special thanks are due to  Shane Kelly for inspiring discussions and suggesting the approach which lead to Lemmas~\ref{lem:slocNormalForm} and \ref{lem:NewHKK4.4}.
We benefited from many helpful discussions with Bhargav Bhatt, Kei Hagihara, Lars Hesselholt, Marc Hoyois, and David Zureick-Brown. We also thank the referee of an earlier version of this paper for pointing out numerous improvements both mathematically and to the exposition. 
The first named author is grateful to Wiesia Nizio\l{} for continued support and encouragement. Finally, we would  like to thank the University of Freiburg, and our home institutions for making key visitations possible. 

\section{Preliminaries and notation}

Throughout, all schemes are assumed to be separated and any reference to a base scheme $S$ assumes that $S$ is Noetherian. For a fixed base scheme $S$ we denote by $\Sch(S)$ the category of schemes of finite type over $S$. When $S = \Spec R$ is affine, we write $\Sch(R)$ for $\Sch(S)$. We denote by $\Reg(S)$ the full subcategory of $\Sch(S)$ of regular schemes, by $\Sch(S)^{\ess}$ the category of schemes essentially of finite type,  by $\Dvr(S)$ the category of schemes essentially of finite type which are regular, local, and of dimension at most $1$, and by $\Shdvr(S)$ the category of schemes of the form $\Spec R^{\sh}$ where $\Spec R$ is in $\Dvr(S)$ and $R^{\sh}$ denotes its strict henselization. 

\subsection{Grothendieck topologies}

A Grothendieck topology on $\Sch(S)$ (or one of the other categories introduced above) is defined by identifying which families of morphisms in $\Sch(S)$ are to be considered as open covers. For a Grothendieck topology $\tau$ denote by $\Sch(S)_{\tau}$ the associated site of $\Sch(S)$. A presheaf $\F$ is a sheaf for the topology $\tau$, if it satisfies the sheaf condition for every covering family in $\Sch(S)$. Moreover, one can associate to a presheaf $\F$ a unique $\tau$-sheaf $\F_\tau$. If $\F$ is already a $\tau$-sheaf, they coincide.

For two Grothendieck topologies $\tau$ and $\mu$ on $\Sch(S)$ one calls $\tau$ {\bf finer} than $\mu$ if every $\mu$-cover is a $\tau$-cover. We denote this relationship by $\mu \rightarrow \tau$. It induces a morphism of associated sites 
	$$ 
	\rho \colon \Sch(S)_{\tau} \to \Sch(S)_{\mu}
	$$ 
and hence a morphism of topoi, called the change of topology map and denoted by
	$$
	(\rho^\ast,\rho_\ast): \widetilde{\Sch(S)}_\tau\rightarrow \widetilde{\Sch(S)}_\mu.
	$$
 Furthermore, we say that a $\mu$-sheaf satisfies {\bf $\tau$-descent} for $X\in\Sch(S)$ if the change of topology map induces an isomorphism
	$\F(X) \cong \F_{\tau}(X).$ 
We say that $\F$ satisfies {\bf cohomological $\tau$-descent} for $X$ if the change of topology map induces isomorphisms
	$$
	H^i_{\mu}(X, \F) \cong H^i_{\tau}(X, \F_{\tau})\quad\text{ for all $i \geqslant 0$}.
	$$ 
For a $\mu$-sheaf $\F$ on $\Reg(S)$, we say that it satisfies {\bf regular $\tau$-descent} provided $\F_\tau(X)=\F(X)$ for all $X\in \Reg(k)$.  In this paper, we are mostly interested in the case where $\mu$ is the Zariski topology. 

We introduce now several Grothendieck topologies we will be working with. They are variants of Voevodsky's $\h$-topology introduced in \cite{Voevodsky1996}.  

\begin{definition}
A morphism in $\Sch(S)$ is called a {\bf topological epimorphism}, if it is surjective and the Zariski topology of the target is the quotient topology of the Zariski topology of the source. It is called universal if this property is preserved by any base change. The {\bf $\h$-topology} is the Grothendieck topology on $\Sch(S)$ with coverings universal topological epimorphisms of finite type.
\end{definition}

The $\h$-topology is the coarsest topology finer than the proper and Zariski topology. It is also finer than the \'etale topology. Several intermediate topologies will appear which are generated by $\cdp$-morphisms and some other variant of covering families. A $\cdp$-morphism is a completely decomposed proper morphism, i.e., a proper morphism such that each $x \in X$ has a preimage $u$ with $[k(u) \colon k(x)] = 1$.  The topology generated by  $\cdp$-covers and Zariski covers is called the $\rh$-topology. If we refine the topology by allowing in addition Nisnevich covers we obtain the $\cdh$-topology. The $\eh$-topology is generated by \'etale covers and $\cdp$-morphisms.  A central topology for our considerations is the $\qfh$-topology which is generated by \'etale covers and finite covers. 

It will be helpful to keep the following diagram in mind which illustrates how the topologies mentioned in this paragraph relate to each other. For a detailed discussion of these and other topologies we refer to \cite{GabberKelly2015} especially \cite[Def.~2.5 and Diag.~(6)]{GabberKelly2015}.
	$$
	\xymatrix{&&&& \fin \ar[d]\ar[r] & \prop\ar[dd]\\ 
	&\Zar \ar[d]\ar[r] & \Nis \ar[d]\ar[r] & \et \ar[d]\ar[r] & \qfh \ar[dr] & \\
	\cdp \ar[r] & \rh \ar[r] & \cdh\ar[r] & \eh \ar[rr]&& \h}
	$$
A special form of $\cdp$-morphism is an abstract blow-up. 

\begin{definition}
Let $X\in \Sch(S)$, $f:X'\rightarrow X$ a proper map and $Z\subset X$ a closed subscheme with preimage $E\subset X'$. We say that $(X',Z)$ is an {\bf abstract blow-up of $X$}, if $f$ induces an isomorphism $X'\backslash E\xrightarrow{\sim} X\backslash Z$ outside $Z$. We denote an abstract blow-up either by $(X'\xrightarrow{f} X, Z\xrightarrow{i} X)$ or more simply by $(X', Z)$. 
\end{definition}

The $\cdp$-topology is in fact generated by abstract blow-ups. Thus the $\rh$-, $\cdh$- and $\eh$-topologies are generated by abstract blow-ups and Zariski, Nisnevich or \'etale coverings respectively, which becomes evident in certain factorizations of covers (e.g. \cite[Thm.~3.1.1]{Voevodsky1996}). If one has resolution of singularities one can even replace abstract blow-ups by blow-ups along smooth centers \cite{CortinasHasemeyerSchlichtingWeibel2005}. 

As abstract blow-ups are $\h$-maps, we see that all proper birational maps are $\h$-covers. Hence, if resolution of singularities holds, all schemes  are $\h$-locally smooth.  But this is even true over any field without the assumption of resolution of singularities by de Jong's alteration theorem, because finite maps and in particular alterations are $\h$-morphisms.

\subsection{Sheafifications of differential forms}

The work of Huber--J\"order and Huber--Kebekus--Kelly describes descent for differential forms. For a fixed field $k$ denote by $\Omega^1$ the presheaf of K\"ahler differential forms on $\Sch(k)$ and for each positive integer $n$ denote by $\Omega^n$ its $n$-fold exterior power. Assume $k$ has characteristic $0$. For the relationships $\cdh \to \eh \to \h$ of topologies, one has descent isomorphisms
	$$
	 \Omega_{\cdh}^n(X)\cong \Omega_{\eh}^n(X) \cong \Omega_{\h}^n(X)
	$$ 
by \cite[Cor.~2.8]{HuberJoerder2014} and  \cite[Thm.~3.6]{HuberJoerder2014} for any  $X\in \Sch(k)$. If $X$ is in addition smooth, one has a descent isomorphism	
	$$
	\Omega^n(X) \cong \Omega_{\eh}^n(X)
	$$ 
by \cite[Thm.~4.7]{Geisser2006} which depends on resolution of singularities. Thus in characteristic $0$ the $\h$-sheafification of K\"ahler differentials agrees with the usual ones in the smooth case. For a klt-base space $X$, one also has an isomorphism  
	$$
	\Omega_{\h}^n(X) \cong (\Omega^{[n]}(X))
	$$ 
between $\h$-differential forms and the push forward of usual differential forms from the smooth locus, called reflexive differentials, \cite[Thm.~5.4]{HuberJoerder2014}. 

When $k$ is of characteristic $p > 0$ the situation is more complicated. For example, an $\h$-sheafification $\F_{\h}$ of a non-trivial presheaf $\F$ can be zero.  This happens in particular for the sheaf of K\"ahler differentials. However some descent statements can be recovered at least for the $\cdh$-topology  via an auxiliary sheaf constructed using the right Kan extension along the inclusion $\Reg(k)\rightarrow \Sch(k)$ denoted $(-)_{\dvr}$. Notably, when $X \in \Reg(S)$ one has $\F_{\dvr}(X) = \F(X)$ and when $k$ has characteristic $0$ one can show 
	$$
	\Omega^n_{\dvr} \cong \Omega_{\h}^n
	$$ 
using \cite{HuberJoerder2014}. For a perfect field $k$ of any characteristic one has isomorphisms
	$$
	\Omega^n(X)\cong  \Omega^n_{\rh}(X) \cong \Omega^n_{\cdh}(X) \cong \Omega^n_{\eh}(X) \cong \Omega^n_{\dvr}(X)
	$$
for a $X\in \Reg(k)$ \cite[Thm.~5.11]{HuberKebekusKelly2016}.

\subsection{The de~Rham--Witt complex} 

Fix a perfect field $k$ of characteristic $p>0$. For a scheme $X$ over $k$, denote by $W \Omega^{\sbt}_X$ the $p$-typical de~Rham--Witt complex  defined by Illusie \cite{Illusie1979}. Set $W \Omega^{\sbt}_{X,\QQ} := W \Omega^{\sbt}_X \otimes \QQ$. It is a complex of \'{e}tale sheaves. For a fixed integer $n$, we consider the sheaf $ W \Omega^{n}_{\QQ}$  of differentials of degree $n$ of the de~Rham--Witt complex.  For short, we refer to it as the sheaf of rational Witt differentials. For a topology $\tau$ on $\Sch(k)$, finer than the \'{e}tale topology we denote as above  by $W \Omega^n_{\QQ,\tau}$ the $\tau$-sheafification which we call {\bf $\tau$-Witt differentials}.

\section{Cohomological descent for Witt differentials assuming resolution of singularities}\label{sec:res}

In this section fix a perfect field $k$ of characteristic $p>0$ and assume strong resolution of singularities. By this we mean that integral and separated schemes over $k$ have smooth proper birational covers and for each proper birational cover $f \colon Y \to X$ where $X$ is smooth, there is a sequence of blow-ups along smooth centres $X_n \to \cdots \to X_1 \to X$ such that $X_n \to X$ factors through $f$. This is the {\it only} section which relies on resolution of singularities. 

Under these assumptions we prove cohomological $\h$-descent for $W\Omega^n_{\QQ}$ on a smooth $k$-variety $X$, i.e., for any $i,n\geqslant 0$ one has 
	$$
	H^i_{\Zar}(X,W\Omega^n_{\QQ}) \cong H^i_{\h}(X,W\Omega^n_{\QQ,\h}).
	$$ We break this up into two steps. The first one is to obtain an (integral) $\eh$-descent under the same assumptions.  

We will use the following lemma which compares \v{C}ech and derived functor cohomology of quasi-coherent sheaves which should not be surprising. We call a $\tau$-sheaf $\F$ for  a Grothendieck topology $\tau$ finer than the Zariski topology quasi-coherent, if  the restriction of $\F$ to all small Zariski sites is quasi-coherent. 
	
\begin{lemma}\label{Lem:Cech}
Let $X$ be a smooth $k$-variety and $\F$ a presheaf on $\Reg(k)$. 
	\begin{enumerate}
	\item If $\F$ is a quasi-coherent \'etale sheaf, its \'etale sheaf cohomology is computed by \v{C}ech cohomology, 
		$$
		H^i_{\et}(X,\F)\cong\varinjlim_{\mathfrak{U}\in\et(X)}\check{H}^i(\mathfrak{U},\F).
		$$ 
	\item If $\F$ is a quasi-coherent $\cdh$-sheaf which satisfies $\cdh$-cohomological descent on $\Reg(k)$, then its $\cdh$-sheaf cohomology is computed by \v{C}ech cohomology, 
		$$
		H^i_{\cdh}(X,\F)\cong\varinjlim_{\mathfrak{U}\in\cdh(X)}\check{H}^i(\mathfrak{U},\F).
		$$
	\item  If $\F$ is a quasi-coherent $\eh$-sheaf which satisfies \'etale and $\cdh$-cohomological descent, then its $\eh$-sheaf cohomology is computed by \v{C}ech cohomology, i.e., 
		$$
		H^i_{\eh}(X,\F)\cong\varinjlim_{\mathfrak{U}\in\eh(X)}\check{H}^i(\mathfrak{U},\F).
		$$
	\end{enumerate}
\end{lemma}
 \begin{proof}
 In the \'etale case it is clear that any $k$-variety  may be covered by affines such that all finite intersections are affine as well. But on affine schemes the higher \'etale sheaf  cohomology groups of a quasi-coherent  sheaf vanish. Hence by Leray's theorem the first statement follows. 
 
 We argue similarly in the $\cdh$-case. Under resolution of singularities, the $\cdh$-topology on $\Reg(k)$ is generated by Nisnevich covers and smooth blow-up squares \cite[Lem.~4.6]{Voevodsky2000}. In particular, any $\cdh$-cover in $\Reg(k)$ has a refinement by smooth affine schemes such that every finite intersection is again smooth and affine. By $\cdh$-descent and the Zariski case
 	$$
 	H^i_{\cdh}(\Spec A, \F_{\cdh})\cong H^i_{\Zar}(\Spec A,\F)=0
 	$$ 
for $i\geqslant 1$ and any smooth affine $k$-variety $\Spec A$. Again by Leray's theorem the second statement follows.

We come now to the case of $\eh$-cohomology. Similar to the $\cdh$-topology, the $\eh$-topology on $\Reg(k)$ is under resolution of singularities generated by \'etale covers and smooth blow-up squares \cite[Cor.~2.6]{Geisser2006}. Hence we may again refine an $\eh$-cover of $X\in\Reg(k)$ by smooth affine schemes $\mathfrak{U}\rightarrow X$ such that all intersections are again smooth affine. We compute the higher \v{C}ech cohomology groups for an arbitrary open affine subset $U\in \mathfrak{U}$. By \textit{loc. cit}  every $\eh$-cover of $U$ has under resolution of singularities a refinement of the form $\{V_i\rightarrow U'\rightarrow U\}$, where $\{V_i\rightarrow U'\}$ is an \'etale cover and $U'\rightarrow U$ is a $\cdh$-map with $U'\in\Reg(k)$. We may refine the cover further such that the $V_i$ are affine and assume without loss of generality that the \'etale cover $\{V_i\rightarrow U'\}$ is non-trivial.

By \cite[Lem.~3.4.3]{Kelly2012} we obtain a first quadrant spectral sequence converging to the \v{C}ech cohomology of the cover $\{V_i\rightarrow U'\rightarrow U\}$ where the entries on the $E_1$-sheet consist of the \v{C}ech cohomology groups associated to fibre products of the cover $\{V_i\rightarrow U'\}$
	$$
	E_1^{p,q}=\check{H}^q(V_{\sbt}^{\times_U^{p+1}}/{U'}^{\times_U^{p+1}},\F) \Rightarrow \check{H}^{p+q}(V_{\sbt}/U,\F),
	$$
where we denote by $V_{\sbt}^{\times_U^{p+1}}$  the $p+1$-fold fiber product $V_{\sbt} \times_U \cdots \times_U V_{\sbt}$. 
As we have seen above the left hand side is just \'etale cohomology of $\F$ on ${U'}^{\times_U^{p+1}}$ and hence by assumption Zariski cohomology of $\F$ on ${U'}^{\times_U^{p+1}}$. In fact, if we fix $U'$ and vary  $\{V_i\rightarrow U'\}$ the left-hand side does not change. By the theory of spectral sequences it follows that the same is true for the right hand side of the spectral sequence. Thus passing to the limit over covers $\{V_i\rightarrow U'\rightarrow U\}$ where $\{V_i\rightarrow U\}'$ is a Zariski cover gives the same result as  passing to the limit over covers $\{V_i\rightarrow U'\rightarrow U\}$ where $\{V_i\rightarrow U'\}$ is an \'etale cover, with $U'\in\Reg(k)$ in both cases.

Note that taking $\{V_i\rightarrow U'\}$  to be  a Zariski cover makes $\{V_i\rightarrow U'\rightarrow U\}$ into a $\cdh$-cover. 
	\begin{eqnarray*}
	\varinjlim_{\mathfrak{V}\in\eh(U)}\check{H}^i(\mathfrak{V}/U,\F) & \cong & \varinjlim_{U'\in\cdh(U)} \varinjlim_{\mathfrak{V}\in\et(U')}\check{H}^i(\mathfrak{V}/U',\F)\\
	& \cong & \varinjlim_{U'\in\cdh(U)} \varinjlim_{\mathfrak{V}\in\Zar(U')}\check{H}^i(\mathfrak{V}/U',\F) \\
	& \cong & \varinjlim_{\mathfrak{V}\in\cdh(U)}\check{H}^i(\mathfrak{V}/U,\F),
	\end{eqnarray*}
	with $U'\in\Reg(k)$. We combine this with
	$$
	\varinjlim_{\mathfrak{V}\in\cdh(U)}\check{H}^i(\mathfrak{V}/U,\F)\cong H^i_{\cdh}(U,\F) \cong H^i_{\Zar}(U,\F),
	$$
where the first equality is (ii) and  the second one comes from the hypothesis that $\F$ satisfies cohomological $\cdh$-descent.  Since $\F$ is quasi-coherent and $U$ affine, this last cohomology group vanishes for $i>0$, in other words $\varinjlim_{\mathfrak{V}\in\eh(U)}\check{H}^i(\mathfrak{V}/U,\F)=0$ for $i>0$ and each $U\in\mathfrak{U}$. Cartan's theorem \cite[Thm.~5.9.2]{Godement1958} which applies in our context mutatis mutandis, implies now that the natural morphisms 
	$$
	\varinjlim_{\mathfrak{U}\in\eh(X)}\check{H}^i(\mathfrak{U}/X,\F)\xrightarrow{\sim} H^i_{\eh}(X,\F)
	$$ 
are isomorphisms for all $i$. 
\end{proof}

\begin{remark}
As far as the authors know there is no obvious direct relationship between the question whether a topos is hypercomplete in the sense of \cite{Lurie2009} and the question whether sheaf cohomology on that topos is computed by \v{C}ech cohomology which was one of the reasons to include the above lemma.  Consider for example the \'etale topos of a finite field which is not hypercomplete according to \cite[War.~7.2.2.31]{Lurie2009}  while in this case \'etale \v{C}ech cohomology computes \'etale sheaf cohomology. On the other hand, the Zariski topos of a Noetherian scheme of finite Krull dimension is hypercomplete \cite[\S~7.2.4]{Lurie2009}, while here \v{C}ech cohomology and sheaf cohomology don't coincide in general.
\end{remark}
	
We can now show the following cohomological descent result.

\begin{proposition}
Let $X$ be a smooth $k$-scheme. Then for all $i,n\geqslant 0$ the change of topology map induces isomorphisms of cohomology groups
	$$
	H^i(X,W\Omega^n)\cong H^i_{\eh}(X,W\Omega^n_{\eh}).
	$$ 
\end{proposition}

\begin{proof}
To begin with, we show  that $W\Omega^n$ satisfies  cohomological $\cdh$-descent on $\Reg(k)$. Under resolution of singularities, the $\cdh$-topology on $\Reg(k)$ is generated by Nisnevich covers and smooth blow-up squares \cite[Lem.~4.6]{Voevodsky2000}. Moreover, \cite[Cor.~3.9]{CortinasHasemeyerSchlichtingWeibel2005} applies under resolution of singularities regardless of the characteristic of the base field.  It  asserts that $W\Omega^n$ satisfies cohomological $\cdh$-descent  on $\Reg(k)$ if and only if it satisfies cohomological Nisnevich descent and if it takes smooth blow-up squares to long exact sequences of cohomology. 

Since $W\Omega^n$ satisfies cohomological \'etale and hence Nisnevich descent it remains to check the condition on smooth blow-up squares. But this follows by a result due to Gros \cite[IV. Thm.~1.1.9]{Gros1985} which gives for a smooth blow-up $X'\rightarrow X$ along a smooth center $Z$ with pull back $Z'=Z\times_X X'$ a long exact sequence
 	$$
 	\cdots\rightarrow H^i(X,W\Omega^n)\rightarrow H^i(X',W\Omega^n)\oplus H^i(Z,W\Omega^n)\rightarrow H^i(Z',W\Omega^n)\rightarrow\cdots.
 	$$ 
Therefore we have the following isomorphisms for $X\in\Reg(k)$ and  $i\geqslant 0$
	\begin{equation}\label{Equ:cdh-descent}
	H^i_{\Zar}(X,W\Omega^n) \cong H^i_{\et}(X,W\Omega^n) \cong H^i_{\cdh}(X,W\Omega_{\cdh}^n).
	\end{equation}
In particular, taking $i=0$, we see that $W\Omega^n_{\cdh} = W\Omega^n$ on $\Reg(k)$. 

Furthermore, we observe, that $W\Omega^n$ is in fact an $\eh$-sheaf. Here again we use the fact that under resolution of singularities every $\eh$-cover of $X\in\Reg(k)$ has a refinement of the form $\{U_i\rightarrow X'\rightarrow X\}$, where $\{U_i\rightarrow X'\}$ is an \'etale cover and $X'\rightarrow X$ is a composition of smooth blow-ups \cite[Cor.~2.6]{Geisser2006}. In particular, $X'\rightarrow X$ is a $\cdh$-morphism in $\Reg(k)$. Hence it follows now  by  \cite[Prop.~3.4.8(3)]{Kelly2012} with $\tau = \eh$, $\rho = \cdh$, and $\sigma = \et$ that $W\Omega^n_{\eh}=W\Omega^n$. 

It remains to show  descent in higher cohomological degrees.  We once again use that  any $\eh$-cover in $\Reg(k)$ has under resolution of singularities a refinement  $\{U_i\rightarrow X'\rightarrow X\}$, where $\{U_i\rightarrow X'\}$ is an \'etale cover and $X'\rightarrow X$ is a $\cdh$-morphism in $\Reg(k)$. While keeping in mind that $W\Omega^n$ is an $\eh$-sheaf  we make the following computation of \v{C}ech cohomology groups using a similar informal notation as in the proof of Lemma~\ref{Lem:Cech}(iii) 
	\begin{eqnarray*}
	\varinjlim_{\mathfrak{U}\in\eh(X)}\check{H}^i(\mathfrak{U}/X,W\Omega^n) & \cong & \varinjlim_{X'\in\cdh(X)} \varinjlim_{\mathfrak{U}\in\et(X')}\check{H}^i(\mathfrak{U}/X',W\Omega^n)\\
	& \cong & \varinjlim_{X'\in\cdh(X)}\varinjlim_{\mathfrak{U}\in\Zar(X')}\check{H}^i(\mathfrak{U}/X',W\Omega^n) \\
	& \cong & \varinjlim_{\mathfrak{U}\in\cdh(X)}\check{H}^i(\mathfrak{U}/X,W\Omega^n),
	\end{eqnarray*}
with $X'\in\Reg(k)$ where the second equality follows similarly as in the proof of Lemma~\ref{Lem:Cech}. 
This implies
	\begin{equation}\label{Eq:eh-descent}
	H^i_{\eh}(X,W\Omega_{\eh}^n) \cong H^i_{\cdh}(X,W\Omega_{\cdh}^n),
	\end{equation}
because	by Lemma~\ref{Lem:Cech} all cohomology groups considered are computed as colimits of \v{C}ech cohomology along covers in the respective topologies.
Putting (\ref{Equ:cdh-descent}) and (\ref{Eq:eh-descent}) together, we obtain $H^i_{\eh}(X,W\Omega^n_{\eh}) = H^i_{\Zar}(X,W\Omega^n)$.
\end{proof}

\begin{remark}
	(i)  One can also use the strategy of \cite[Thm.~4.3]{Geisser2006} to show the above result. The main point is to compare long exact blow-up sequences for Zariski and $\eh$-topology and then make a spectral sequence argument as in the proof of  \cite[Proof of Thm.~3.1]{Kohrita2017}. \\
	(ii)  In \cite[Thm.~3.5]{GeisserHesselholt2009} Thomas~Geisser and Lars~Hesselholt prove  under resolution of singularities a version of above's result for Witt differentials with finite coefficients. For $X\in\Reg(k)$ they deduce the isomorphisms
		$$
		H^i(X,W_m\Omega^n) \cong H^i_{\cdh}(X,W_m\Omega_{\cdh}^n) \cong H^i_{\et}(X,W_m\Omega^n) \cong H^i_{\eh}(X,W_m\Omega^n).
		$$
by induction on $m$ from \cite[Cor.~3.9]{CortinasHasemeyerSchlichtingWeibel2005} which was also used in the proof above.
\end{remark}	

We come now to the second part of the proof of the main result of this section.

\begin{proposition}\label{prop:eh-h}
For a smooth $k$-scheme $X$ the change of topology map induces for all $i,n\geqslant 0$ isomorphisms
	$$
	H^i_{\eh}(X, W\Omega^n_{\QQ,\eh}) \cong H^i_{\h}(X, W\Omega^n_{\QQ,\h}).
	$$
\end{proposition}

\begin{proof}
Note that \cite[Prop.~6.1]{HuberJoerder2014}, whose proof does not depend on the characteristic of the base field, implies immediately that  $H^i_{\h}(X,W\Omega^n_{\QQ,\h}) \cong H^i_{\eh}(X,W\Omega^n_{\QQ,\h})$. Thus the desired statement follows if one shows $W\Omega^n_{\QQ,\eh}\cong W\Omega^n_{\QQ,\h}$. 

This follows by  the same argument used in the proof of the first statement of \cite[Thm.~3.6]{HuberJoerder2014}.  Namely, we need to check that for any $X\in\Sch(k)$, the canonical morphism
	$$
	W\Omega^n_{\QQ,\eh}(X) \rightarrow \varinjlim_{\mathfrak{U}\in\h(X)}\check{H}^0(\mathfrak{U},W\Omega^n_{\QQ,\eh})
	$$
where $\mathfrak{U}$ runs over all $h$-covers of $X$ is an isomorphism. In fact it is enough to show the sheaf condition for $W\Omega^n_{\QQ,\eh}$ for a refinement of an arbitrary $\h$-cover $\mathfrak{U}$. 

By \cite[Thm.~3.1.9]{Voevodsky1996} any $\h$-cover of $X$ has a refinement of the form $\{U_i\rightarrow \bar{U}\rightarrow X'\rightarrow X\}$ with a blow-up $X'\rightarrow X$ in a closed subvariety, a finite surjective map $\bar{U}\rightarrow X'$, and a Zariski cover $\{U_i\rightarrow \bar{U}\}$. By resolution of singularities we may assume that $X'$ is smooth and $X'\rightarrow X$ a sequence of smooth blow-ups.  For every connected component $X'_j$ of $X'$ choose an irreducible component $\bar{U}_j$ of $\bar{U}$ which maps surjectively to $X'_j$. Replace $\bar{U}$ by the disjoint union of the normalizations of $\bar{U}_j$ in the normal hull of $k(\bar{U}_j)/k(X'_j)$ thereby refining the cover further to one of the form $\left\{U_{ij}\rightarrow \bar{U}_j\rightarrow  \cup X'_j=X'\rightarrow X\right\}$ such that for $j$ fixed $\{U_{ij}\rightarrow \bar{U}_j\}_i$  is an open Zariski cover, the $X'_j$ are the  irreducible components of $X'$,  for each $j$, $\bar{U}_j \rightarrow X'_j$ is a finite surjective map of irreducible normal varieties, and $X'\rightarrow X$ is as above. As in \cite[Thm.~3.6]{HuberJoerder2014}, it suffices to check the sheaf condition for $W\Omega^n_{\QQ,\eh}$ for each of the three intermediate maps. 

Since the first and the last map in the refinement are in particular $\eh$-covers, for which the sheaf condition for $W\Omega^n_{\QQ,\eh}$ holds trivially, it remains to show that $W\Omega^n_{\QQ,\eh}$ satisfies the sheaf condition for a finite surjective map $\bar{U}\rightarrow X'$ of irreducible normal varieties. Although this is checked on strict henselisations of discrete valuation rings in the proof of Proposition~\ref{prop:qfhDescent} for $W\Omega^n_{\QQ}$, the same proof holds for schemes as well. Furthermore,  it is clear that the sheaf condition then holds for  $W\Omega^n_{\QQ,\eh}$, too.

As the presented argument depends on later results, we note that nothing outside of Section 3 involves this proposition.
\end{proof}

\begin{corollary}
Let $X\in\Reg(k)$. Then one has for all $i,n\geqslant 0$ isomorphisms
	$$
	H^i_{\Zar}(X, W\Omega^n_{\QQ}) \cong H^i_{\h}(X, W\Omega^n_{\QQ,\h}).
	$$
\end{corollary}

We cannot yet establish cohomological $\h$-descent for $W \Omega_{\QQ}^n$ without assuming resolution of singularities. The reason is that both approaches presented above use the refinement of an $\eh$-cover into an \'etale cover and a series of smooth blow-ups given by Geisser. To obtain such a refinement strong resolution of singularities is needed. 	If one attempts to use alterations here instead finite maps are introduced in the process, which cannot be controlled.

\section{$dvr$-Witt differentials}\label{sec:dvr}

In this section we apply the extension functor defined in  \cite[Sec.~4.1]{HuberKebekusKelly2016} to Witt differentials. It can be interpreted as the right adjoint of the restriction functor from presheaves on $\Sch(S)$ to presheaves on $\Reg(S)$, in other words the right Kan extension along the inclusion $\Reg(S)\rightarrow \Sch(S)$. As we will see, it is a useful tool to obtain descent properties on regular schemes.

\subsection{The extension functor on Witt differentials} 

We briefly recall the definition of the extension functor and important properties. For more details we refer to \textit{loc. cit.}. 

\begin{definition}
For a presheaf $\F$ on $\Reg(S)$, let $\F_{\text{dvr}}$ be the presheaf on $\Sch(S)$ given by
	$$
	\F_{\text{dvr}}(X):=\varprojlim_{Y\in\Reg(X)}\F(Y)
	$$
for $X\in\Sch(S)$. 
\end{definition} 

Huber, Kebekus, and Kelly show that if $\F$ is unramified in the sense of \cite[Def.~2.1]{Morel2012}, then  $\F_{\text{dvr}}$ can be expressed in a form which is particularly useful to establish descent results.
\begin{definition}\label{def:Unramified} A presheaf $\F$ on $\Reg(S)$ is {\bf unramified} if it is a Zariski sheaf and satisfies the following  conditions for all $X$ and $Y$ in $\Reg(S)$.
	\begin{enumerate}
	\item The natural morphism $\F( X \sqcup Y) \to \F(X) \times \F(Y)$ is an isomorphism.
	\item For a dense open immersion $U \to X$, the induced map $\F(X) \to \F(U)$ is injective.
	\item For every open immersion $U \to X$ containing all points of codimension at most $1$,  the induced map $\F(X) \to \F(U)$ is an isomorphism. 
	\end{enumerate}
A presheaf $\F$ on $\Sch(S)$ is unramified if its restriction to $\Reg(S)$ is unramified.
\end{definition}

If $\F$ is an unramified presheaf on $\Sch(S)$, then $\F_{\dvr}$ can be expressed as
	\begin{equation}\label{DvrSheaf}
	\F_{\dvr}(X) \cong \varprojlim_{W\in\Dvr(X)}\F^{\ess}(W)
	\end{equation}
according to \cite[Prop.~4.14]{HuberKebekusKelly2016}. Here $\F^{\ess}$ denotes the left Kan extension of the presheaf $\F$ to a presheaf on $\Sch(S)^{\ess}$ of $S$-schemes essentially of finite type.

Note that this allows us to view the extension functor  $(-)_{\dvr}$ as the right Kan extension along the inclusion $\Dvr(S) \to \Sch(S)^{\ess}$. We apply it now to the sheaves $W\Omega^n_{\QQ}$, $n\geqslant 0$, over a perfect field $k$ of positive characteristic $p$. 

\begin{proposition}\label{prop:Unramified}
Let $k$ be a perfect field of characteristic $p>0$. Both the presheaf $W\Omega^n$ and the presheaf $W\Omega_{\QQ}^n$ are unramified on $\Reg(k)$. 
\end{proposition}
\begin{proof} 
It suffices to show that $W \Omega^n$ is unramified. Since $W\Omega^n$ is an \'etale sheaf, it remains to verify the three conditions defining unramified presheaves.  

To begin with, let $X,Y\in\Reg(k)$. We have to show that the natural  map $W\Omega^n(X\sqcup Y)\rightarrow W\Omega^n(X)\times W\Omega^n(Y)$ for each $n\geqslant 0$ is an isomorphism. Since $\{X,Y\}$ is an \'etale cover of $X\sqcup Y$ and $W\Omega^n$ is an \'etale sheaf, we have an exact sequence
	$$
	0\rightarrow W\Omega^n(X\sqcup Y) \rightarrow W\Omega^n(X)\times W\Omega^n(Y)\rightarrow W\Omega^n(X \times_{X\sqcup Y} Y).
	$$
But $X \times_{X\sqcup Y} Y=\varnothing$ and therefore $W\Omega^n(X \times_{X\sqcup Y} Y)=0$, and $W\Omega^n(X\sqcup Y) \rightarrow W\Omega^n(X)\times W\Omega^n(Y)$ is an isomorphism.

Next let $j:U\rightarrow X$ be a dense open immersion in $\Reg(k)$. We will show that for any $n\geqslant 0$ the induced morphism $W\Omega^n_X\rightarrow j_\ast W\Omega^n_U$ is injective. If this is not the case, then we can find an affine (Zariski) open $\Spec A=V\subset X$ such that for some $n\geqslant 0$ there is a nonzero element $\omega\in W\Omega^n(V)$ which maps to zero in $W\Omega^n(U\cap V)$. By restricting further we may assume that $W=U\cap V=\Spec B$ is also affine. Consequently, by \cite[I.Prop.~1.13.1]{Illusie1979} we can  identify $W\Omega^{\sbt}(V)$ with $W\Omega^{\sbt}_A$ and $W\Omega^{\sbt}(U\cap V)$ with $W\Omega^{\sbt}_B$. Moreover, by \cite[I.Thm.~1.3]{Illusie1979} the morphism $W\Omega^{\sbt}_A\rightarrow W\Omega^{\sbt}_B$ corresponds via the functorial isomorphism \cite[I.(1.3.1)]{Illusie1979} to the  map $A=\OO(V)\rightarrow B=\OO(U\cap V)$, which is injective by definition as $j(U)$ is dense in $X$.  It follows then easily that  $W\Omega^{\sbt}_A\rightarrow W\Omega^{\sbt}_B$ has to be injective as well, which is a contradiction to $\omega\neq 0$. This shows the claim.

To complete the proof we now consider  an open immersion $j:U\rightarrow X$ in $\Reg(k)$ containing all points of codimension $\leqslant 1$. Since this is a local question, we reduce without loss of generality to the case that $U$ and $X$ are affine, and so \cite[I.Prop.~1.13.1]{Illusie1979} applies again. In particular, for each $l\geqslant 1$, the complex $W_l\Omega^{\sbt}(X)$ is a quotient of $\Omega^{\sbt}_{W_l\OO(X)}$ and similarly for $U$. Consider the induced morphism of complexes $j:W\Omega^{\sbt}(X)\rightarrow W\Omega^{\sbt}(U)$. Before taking limits, $j:W_{\sbt}\Omega^{\sbt}(X)\rightarrow W_{\sbt}\Omega^{\sbt}(U)$ is a morphism in the category of de~Rham-$V$-pro-complexes defined in \cite[I.~Def.~1.1]{Illusie1979}. In degree zero it is given by the morphism of pro-objects
	$$
	j:W_{\sbt}\OO(X)\xrightarrow{\sim} W_{\sbt}\OO(U)
	$$
which by the definition of Witt vectors is an isomorphism since $\OO(X)\xrightarrow{\sim}\OO(U)$ is an isomorphism. Hence it has an inverse $h:W_{\sbt}\OO(U)\rightarrow W_{\sbt}\OO(X)$, which induces by the universal property of the de~Rham--Witt complex a morphism
	$$
	h:W_{\sbt}\Omega^{\sbt}(U)\rightarrow W_{\sbt}\Omega^{\sbt}(X)
	$$
in the category of de~Rham-$V$-pro-complexes. To see that $j$ and $h$ are inverse to each other in the category of de~Rham-$V$-pro-complexes, observe that for each $l\geqslant 1$, the morphisms $j:W_l\OO(X)\rightarrow W_l\OO(U)$ and $h:W_l\OO(U)\rightarrow W_l\OO(X)$  induce morphisms
	$$
	\xymatrix{\Omega_{W_l\OO(X)}^{\sbt} \ar@<.5ex>[r]^{j} & \Omega_{W_l\OO(U)}^{\sbt} \ar@<.5ex>[l]^{h}}
	$$
which are inverse to each other. By the surjectivity of the projections $\pi_l:\Omega_{W_l\OO(X)}^{\sbt}\rightarrow W_l\Omega^{\sbt}(X)$ and $\pi_l:\Omega_{W_l\OO(U)}^{\sbt}\rightarrow W_l\Omega^{\sbt}(U)$  for all $l\geqslant 1$, the morphisms
	$$
	\xymatrix{W_{\sbt}\Omega^{\sbt}(X)  \ar@<.5ex>[r]^{j} & W_{\sbt}\Omega^{\sbt}(U) \ar@<.5ex>[l]^{h}}
	$$
are inverse to each other, too, and $j$ is an isomorphism. After taking limits, we see that $j:W\Omega^{\sbt}(X)\rightarrow W\Omega^{\sbt}(U)$ is an isomorphism of complexes, in particular $W\Omega^n(X)\rightarrow W\Omega^n(U)$ is an isomorphism for every $n\geqslant0$. 
\end{proof}

We call the presheaf $W\Omega^n_{\QQ,\dvr}$ obtained by applying the extension functor to the presheaf $W\Omega^n_{\QQ}$ of rational Witt differentials on $\Sch(k)$ the {\bf $\mathbf{dvr}$-Witt differentials} of degree $n$.

\begin{corollary}\label{cor:wittdvr}
Let $X\in\Sch(k)$. The $dvr$-Witt differentials of $X$ can be written as
$$W\Omega^n_{\QQ,\dvr}(X) \cong \varprojlim_{Y\in\Dvr(X)} W\Omega^n_{\QQ}(Y).$$
\end{corollary}
\begin{proof} 
Because $W\Omega^n_{\QQ}$ is unramified as we have seen in Proposition~\ref{prop:Unramified},    equation (\ref{DvrSheaf}) applies and we obtain
	$$
	W\Omega^n_{\QQ,\dvr}(X) \cong \varprojlim_{Y\in\Dvr(X)} W\Omega^{n,\mathsf{ess}}_{\QQ}(Y).
	$$ 
	However, by \cite[I.1.10]{Illusie1979} the sheaf $W\Omega^n_{\QQ}$ commutes with filtered colimits which means that $W\Omega^{n,\mathsf{ess}}_{\QQ}(Y) \cong W \Omega^n_{\QQ}(Y)$. 
\end{proof}

\subsection{Descent results for $\dvr$-Witt differentials}\label{subsec:dvrdescent}

In \cite[Lem.~4.4]{HuberKebekusKelly2016}, Huber, Kebekus and Kelly observe that the extension functor preserves sheaves for topologies equal or coarser than the \'etale topology. This is a consequence of the fact that the inclusion $\Reg(S)\rightarrow \Sch(S)$ is cocontinuous for these topologies. Consequently, for a presheaf $\F$ it is enough to check the sheaf conditions for the mentioned topologies on $\Reg(S)$ in order to show that $\F_{\dvr}$ is a sheaf on $\Sch(S)$. 

 We would like to use a similar technique to prove that $W\Omega^n_{\QQ,\dvr}$ is a $\qfh$-sheaf. However as the $\qfh$-topology is finer than the \'etale topology, we have to realise the extension  functor as the right Kan extension along the inclusion $\Shdvr(k)\rightarrow \Sch(k)^{\ess}$ which is cocontinuous for the $\qfh$-topology. More precisely, we have the following.

\begin{lemma}\label{lem:slocNormalForm}
For a perfect field $k$ characteristic $p>0$, the inclusion $\Shdvr(k) \rightarrow \Sch(k)^{\ess}$ induces a  cocontinuous morphism of $\qfh$-sites. 
\end{lemma}
\begin{proof} Let $X$ be  in $\Shdvr(k)$. We have to show that every $\qfh$-cover  $\{ U_i \rightarrow X \}$  in $\Sch(k)^{\ess}$ has a refinement to a $\qfh$-cover in $\Shdvr(k)$. Without loss of generality, we may assume that $X$ is connected.  Moreover, by \cite[Cor.~5.6]{Greco1976}  any object in $\Shdvr(k)$ is excellent. Hence according to \cite[Lem.~10.4]{SuslinVoevodsky1996}, every $\qfh$-cover as above  has a refinement of the form $\{V_j\rightarrow V\rightarrow X\}_{j\in J}$, where $V$ is the normalization of $X$ in a finite normal extension of its function field and $\{V_j \to V\}$ is a Zariski cover. In particular  $V \rightarrow X$ is absolutely flat. With \cite[Thm.~5.3~i)]{Greco1976} we can now deduce that $V$ is also the strict henselisation of a discrete valuation ring. 
\end{proof}

\begin{corollary}\label{cor:qfhSheaf}
For a perfect field $k$ characteristic $p>0$, if $\F$ is a $\qfh$-sheaf on $\Shdvr(k)$, then $\F_{\shdvr}$ is a $\qfh$-sheaf. 
\end{corollary}
\begin{proof}
This follows from the previous lemma analogous to \cite[Lem.~4.4]{HuberKebekusKelly2016}.
\end{proof}

Note that $\Shdvr(k)$ is not a subcategory of $\Dvr(k)$, but the objects of $\Shdvr(k)$ are colimits of \'etale extensions of objects in $\Dvr(k)$. None the less, we will compare $\F_{\dvr}$ to the the right Kan extension along $\Shdvr(k) \to \Sch(k)^{\ess}$ given by 
	$$
	\F_{\shdvr}(X) := \varprojlim_{ Y \in \Shdvr(X)} \F(Y)
	$$ 
for \'etale sheaves $\F$ which commute with filtered colimits.

\begin{lemma}\label{lem:NewHKK4.4}
Let $k$ be a perfect field of characteristic $p>0$. For $X\in \Sch(k)^{\ess}$ and  any \'etale sheaf $\F$ on $\Sch(k)$ which commutes with filtered colimits there is a natural morphism  $\F_{\dvr}(X) \xrightarrow{\sim} \F_{\shdvr}(X)$ which is an isomorphism
\end{lemma}
\begin{proof}
Recall that a sheaf which commutes with filtered colimits satisfies $\F^{\ess} \cong  \F$ as presheaves. We first define a natural transformation $\F_{\dvr} \rightarrow \F_{\shdvr}$. We may write an  object  $W\in \Shdvr(X)$ as  $W = \Spec R^{\sh}$ where $R\in\Dvr(X)$  and  
	$$
	R^{\sh} = \varinjlim_{R \rightarrow A \rightarrow \kappa^s} A
	$$ 
is  the colimit  over all factorizations of $R \rightarrow \kappa^s$ for a separable closure  $\kappa^s$ of the field $\kappa$ corresponding to the generic point of $\Spec R$  such that $R \to A$ is  \'etale. Since $R$ is a valuation ring of finite Krull dimension, it suffices by  \cite[Cor.~2.17]{HuberKelly2017}  to consider the valuation rings $A$ in the above colimit. In this case $A$ is again in $\Dvr(X)$.

By the hypothesis that $\F$ commutes with colimits we have
	$$
	\F(W) = \F(\Spec R^{\sh}) \cong  \varinjlim_{R \rightarrow A \rightarrow \kappa^s} \F(\Spec A).
	$$ 
 Since we have natural projections $\F_{\dvr}(X) \rightarrow \F(\Spec A)$ for each $\Spec A\in \Dvr(X)$, there is a morphism $\F_{\dvr}(X) \to \F(W)$. These maps are compatible by functoriality  as $W$ varies over $\Shdvr(X)$. Consequently  there is a morphism 
 	$$
 	\F_{\dvr}(X) \rightarrow \F_{\shdvr}(X)  \cong  \varprojlim_{W \in \Shdvr(X)} \F(W)
 	$$ 
 which itself is obviously functorial as desired.  

The rest of the proof follows mutatis mutandis from \cite[Prop.~3.8]{HuberKelly2017}, however we explain the critical changes. Fix a section $s$ in $\F_{\dvr}(X)$ mapping to the zero section in $\F_{\shdvr}(X)$. For a geometric point $x \in X$ with residue field $\kappa$ and a separable closure $\kappa^s$, we have $\F(\kappa^s) = \varinjlim_{\lambda/\kappa} \F(\lambda)$  where $\lambda$ runs over all finite subextensions of $\kappa^s/\kappa$. Because $s\big|_{\kappa^s}$ is zero by assumption, there is $\lambda$ such that $s\big|_{\lambda} = 0$. By \'etale descent, $\F(\kappa) \rightarrow \F(\lambda)$ is injective for all $\lambda$, and hence $s\big|_\kappa = 0$ as desired. A similar argument works for valuation rings of $\kappa$ and their strict henselisations. This shows that the morphism $\F_{\dvr}\rightarrow \F_{\shdvr}$ defined above is injective.

 Now fix $t \in \F_{\shdvr}(X)$. For any $Y = \Spec R \to X$ let $R^{\sh}$ be its strict henselisation. The element $t_{R^{\sh}}$ is fixed by $\Gal(\Frac(R^{\sh})/\Frac(R))$, so lifts to $\F(R)$ in a compatible way to define a preimage. 
\end{proof}

\begin{remark}
The proofs of Lemma~\ref{lem:slocNormalForm} and Lemma~\ref{lem:NewHKK4.4} also work over Nagata base schemes, but this will not be needed. 
\end{remark}

With these facts established we can now prove $\qfh$-descent for the extension $W\Omega^n_{\QQ,\dvr}$ of the sheaf of Witt differentials following an approach used in the study of regular group schemes \cite[Prop.~C.2]{AnconaHuberPepinLehalleur2016}.

\begin{proposition}\label{prop:qfhDescent}
 For a perfect field $k$ of characteristic $p > 0$, the extension $W\Omega^n_{\QQ,\dvr}$ is a $\qfh$-sheaf on $\Sch(k)$.
\end{proposition}

\begin{proof}
By Lemma~\ref{lem:NewHKK4.4}, $W\Omega_{\QQ,\dvr}^n \cong  W\Omega_{\QQ,\shdvr}^n$ and thus it is equivalent to show that $W\Omega_{\QQ,\shdvr}^n$ is a $\qfh$-sheaf. But according to Corollary~\ref{cor:qfhSheaf} this follows if we show that $W\Omega^n_{\QQ}$ is a $\qfh$-sheaf on $\Shdvr(k)$. 

To see this we check the sheaf condition for a $\qfh$-cover $\{U_i\rightarrow X\}_{i\in I}$ in $\Shdvr(k)$. Without loss of generality, we may assume that $X$ is connected and therefore integral as it is Noetherian. As in the proof of Lemma~\ref{lem:slocNormalForm}, we may replace this with a refinement of the form $\{V_j\rightarrow V\rightarrow X\}_{j\in J}$, where $V$ is the normalization of $X$ in a finite normal extension of its function field. In particular, $V\rightarrow X$ is finite surjective and $\{V_j\rightarrow V\}$ is a Zariski cover.  As $W\Omega^n_{\QQ}$ is a Zariski sheaf,  it suffices to show that the  sheaf condition holds for the cover $\pi:V\rightarrow X$.

To begin with, we argue as in \cite{AnconaHuberPepinLehalleur2016} that $W\Omega^n_{\QQ}\big|_{\Shdvr(k)}$ is separated for the $\qfh$-topology. To see this we first observe that for a dominant morphism $f:X\rightarrow Y$ the induced morphism $W\Omega_{\QQ}^n(Y)\rightarrow W\Omega_{\QQ}^n(X)$ is injective, because $\OO(Y)\rightarrow\OO(X)$ is, and then apply a result due to Voevodsky \cite[Prop.~3.1.4]{Voevodsky1996}. 

Let $\widetilde{V}$ be the normalization of the reduction of $V\times_X V$ and consider the sequence
\begin{equation}\label{ExactSeq?}
0\rightarrow W\Omega^n_{\QQ}(X)\rightarrow W\Omega^n_{\QQ}(V) \rightarrow W\Omega^n_{\QQ}(\widetilde{V})
\end{equation}
where the rightmost map is induced by the difference of the two projections $\widetilde{V}\rightarrow V$. Our goal is to show this sequence is exact. 

We may find a finite surjective morphism $\pi':Y \rightarrow V$ with $Y$ normal, such that the composition with $\pi:V\rightarrow X$ factors into a generically purely inseparable morphism followed by a generically Galois morphism $\pi'\pi=\pi_s\pi_i.$ Since $W\Omega_{\QQ}^n$ is $\qfh$-separated, the sheaf condition for $\pi'\pi$ follows if we can verify the sheaf condition for $\pi_i$ and $\pi_s$ separately, and this on the other hand implies the sheaf condition for $\pi$.

Thus assume first that  $\pi:V\rightarrow X$ is a generically Galois morphism in $\Shdvr(k)$ with Galois group $\Gal(\kappa(V)/\kappa(X))=\Gamma$. By definition $V$ is the normalization of $X$ in $\kappa(V)$ and the action of $\Gamma$ extends by functoriality from $\kappa(V)$ to  $V$. Furthermore, the canonical map  $\OO_X\xrightarrow{\sim}(\pi_\ast\OO_V)^\Gamma$ is an isomorphism. With $X=\Spec A$ and $V=\Spec B$ we observe that $\pi$ corresponds to an extension $A\rightarrow B$ with $A=B^\Gamma$, where $A$ and $B$ are strict henselisations of discrete valuation rings, hence in  particular valuation rings, and we have a commutative diagram
	$$
	\xymatrix{A \ar[r]^\pi \ar[d] & B \ar[d]\\ k(X) \ar[r]^\pi & k(V)}
	$$
Our goal is to show that the sequence (\ref{ExactSeq?}) is exact in this case. Since $W\Omega^n_{\QQ}$ is an \'etale sheaf and therefore satisfies Galois descent, we deduce immediately that that the sequence
	\begin{equation}\label{ExactSeq!}
	0\rightarrow W\Omega^n_{\QQ}(k(X))\rightarrow W\Omega^n_{\QQ}(k(V)) \rightarrow W\Omega^n_{\QQ}(k(V)\otimes_{k(X)}k(V))
	\end{equation}
is exact. Moreover, the canonical morphisms induce a commutative diagram 
	$$
	 \xymatrix{W\Omega^n_{\QQ}(X) \ar@{^{(}->}[d] \ar[r]^\pi & W\Omega^n_{\QQ}(V)  \ar@{^{(}->}[d] \\
	 W\Omega^n_{\QQ}(k(X)) \ar@{^{(}->}[r]^\pi & W\Omega^n_{\QQ}(k(V))  }
	 $$
where the vertical maps are injective because $W\Omega^n_{\QQ}$ is unramified. It follows that $W\Omega^n_{\QQ}(X)\xrightarrow{\pi} W\Omega^n_{\QQ}(V)$ is injective too, and that (\ref{ExactSeq?}) is exact at the first term. 

To show that it is exact at the second term, let $\omega$ be a section  in the kernel of $W\Omega^n_{\QQ}(V) \rightarrow W\Omega^n_{\QQ}(\widetilde{V})$. We have to show that it has a preimage under $\pi$. Its restriction $\omega_{k(V)}$ to  $\Spec k(V)$, i.e., its image under the canonical injective map $W\Omega^n_{\QQ}(V) \hookrightarrow  W\Omega^n_{\QQ}(k(V))$ has a preimage $\upsilon_{k(X)}$ under $\pi$ by the short exact sequence (\ref{ExactSeq!}). But $\upsilon_{k(X)}$ is in fact contained in $W\Omega^n_{\QQ}(k(X))\cap W\Omega^n_{\QQ}(V)$ and as $V\rightarrow X$ corresponds to a Galois extension of valuation rings, one can show by an inductive argument using  the short exact sequence of \cite[I.3.15]{Illusie1979} that  it lifts indeed to a section $\upsilon$ of $W\Omega^n_{\QQ}(X)$ over the valuation ring $A$. This is a preimage of $\omega$ under $\pi$ as desired.

Now we look at the case, when $\pi$ is generically purely inseparable. As $X$ is in $\Shdvr(k)$, $\pi$ is actually purely inseparable as there is only one fiber. Hence its diagonal map is a surjective closed immersion. After reduction, we may assume that it is an isomorphism
$\Delta(\pi): V^{\text{red}}\xrightarrow{\sim} (V\times_X V)^{\text{red}}=\widetilde{V}.$
The map $W\Omega^n_{\QQ}(V)\rightarrow W\Omega^n_{\QQ}(\widetilde{V})$ being induced by the difference of the two projections $\widetilde{V}\rightarrow V$ consequently is the zero map, so that it remains to show that 
$\pi^\ast: W\Omega^n_{\QQ}(X)\hookrightarrow W\Omega^n_{\QQ}(V)$ 
is surjective. If $\pi$ is an isomorphism this is clear. Assume therefore that $\pi$ is not an isomorphism. Similar to \cite{AnconaHuberPepinLehalleur2016} we reduce this to the case where $\pi$ is a relative Frobenius as follows. 

Denote by $V^{(p)}$ the base change of $V$ along the absolute Frobenius $F_k$ of $\Spec(k)$. We have a commutative diagram 
	$$
	\xymatrix{V\ar[d] & V^{(p)} \ar[l]_{W} \ar[d] & V  \ar[dl] \ar[l]_{F_{V/k}} \ar@/_2pc/[ll]_{F_V}\\ k & k\ar[l]^{F_k}}
	$$
where $F_V$ is the absolute Frobenius of $V$, $F_{V/k}$ the relative Frobenius of $V$ over $k$ and $W$ the canonical projection. As $k$ is perfect, $F_k$ is an automorphism and it follows that $V\cong V^{(p)}$ and the relative Frobenius $F_{V/k}$ coincides with the absolute Frobenius $F_V$ up to isomorphism. For a $p$-power $q=p^r$ we obtain a similar diagram by iteration. 

By a result due to Koll\'ar \cite[Prop.~6.6]{Kollar1997} there is a $p$-power $q$ and a morphism $\tilde{\pi}: X\rightarrow V^{(q)}$ such that the relative Frobenius factors as
 $F_{V/k}^{r}=\tilde{\pi}\circ\pi: V\rightarrow X\rightarrow V^{(q)}$.
Thus it suffices to show that $(F_{V/k}^r)^\ast$ is surjective. 

Since $V$ as an irreducible scheme is an $k$-scheme, the Witt vector Frobenius is induced by the absolute Frobenius $F_V^\ast=\Frob: W\Omega^n(V)\rightarrow W\Omega^n(V)$,
which commutes with the Verschiebung map $\Ver$. In particular, $\Frob\Ver=\Ver\Frob=p$. Consequently, after inverting $p$, all of the maps $\Frob=F_V^\ast$, $F_{V/k}^\ast$, and the maps $(F_{V/k}^r)^\ast$ for all $q=p^r$ are surjective. 
\end{proof}

\begin{remark}
One notes from the proof that $W \Omega^n$ is a $\qfh$-sheaf after inverting only the Witt vector Frobenius. In fact, one can use the same argumentation as above for unramified \'etale sheaves $\F$ which commute with colimits such that any finite morphism $V \to X$ of objects in $\Shdvr(k)$ which is generically Galois with Galois group $G$ satisfies $\F(X) \cong \F(V)^G$ and for which the relative Frobenius $F_{V/k}^r$  induces surjections for all $r \geqslant 0$. 
\end{remark}

While our result along with \cite[Lem.~4.4]{HuberKebekusKelly2016}  shows that the extension functor preserves sheaves for certain topologies, Huber, Kebekus, and Kelly prove even a stronger result in \cite[Prop.~4.18]{HuberKebekusKelly2016}. 

\begin{proposition}\label{HKKProp418}
Let $S$ be a Noetherian scheme. If $\F$ is an unramified presheaf on $\Sch(S)$, then $\F_{\dvr}$ is an $\rh$-sheaf. In particular, if $\F$ is an unramified Nisnevich, respectively \'etale sheaf on $\Sch(S)$, then $\F_{\dvr}$ is a $\cdh$-sheaf, respectively $\eh$-sheaf.
\end{proposition}

As a consequence we immediately obtain the following statement.

\begin{corollary}
For a perfect field $k$ of positive characteristic $p$, the presheaves $W\Omega^n_{\dvr}$ and $W\Omega^n_{\QQ,\dvr}$  are $\eh$-sheaves on $\Sch(k)$. 
\end{corollary}

\begin{remark}\label{Rem:4.6}
By the universal property of sheafification there are canonical morphisms
	 $$
	 W\Omega_{\QQ}^n\rightarrow W\Omega_{\QQ,\rh}^n \rightarrow W\Omega_{\QQ,\cdh}^n\rightarrow W\Omega_{\QQ,\eh}^n\rightarrow W\Omega_{\QQ,\dvr}^n.
	 $$ 
Moreover, if $X$ is regular, $W\Omega_{\QQ}^n(X) \cong W\Omega_{\QQ,\dvr}^n(X)$ by \cite[Rem.~4.3.3]{HuberKebekusKelly2016} and the above composition is an isomorphism on $X$. 
\end{remark}

To close this section we put the pieces together and deduce $\h$-descent for $W\Omega^n_{\QQ,\dvr}$.  The following lemma  is well-known to eperts and follows from \cite[Thm.~3.1.1]{Voevodsky1996} and Stein factorization. 

\begin{lemma}\label{lemHRefine} 
Let $X \in \Sch(S)$ for a base scheme $S$. Every $\h$-cover $\{U_i\xrightarrow{p_i} X\}_i$ has a refinement of the form 
$$\{V_i\rightarrow \overline{V}\rightarrow X'\rightarrow X\}_i$$
where $\{V_i\rightarrow \overline{V}\}_i$ is a Zariski cover, $\overline{V}\rightarrow X'$ is a modification and $X'\rightarrow X$ is a finite morphism.
\end{lemma}

It induces a descent result which we hope is of independent interest. 

\begin{proposition}\label{prop:dvrHDescent} 
For a base scheme $S$   let $\F$ be an unramified \'etale sheaf on $\Reg(S)$. If $\F_{\dvr}$ is a $\qfh$-sheaf, then $\F_{\dvr}$ is also an $\h$-sheaf on $\Sch(S)$.
\end{proposition}
\begin{proof}
Any $\h$-cover $\U \to X$ may be refined as $\{ W_i \to X' \to X\}_{i\in I}$ where the first is an $\eh$-cover and the latter is a $\qfh$-cover by Lemma~\ref{lemHRefine}. By hypothesis $(\F_{\dvr})_{\qfh} \cong \F_{\dvr}$. Also the hypothesis and Theorem~\ref{HKKProp418} gives $(\F_{\dvr})_{\eh} \cong \F_{\dvr}$. The theorem now follows by \cite[Prop.~3.4.8(3)]{Kelly2012} with $\tau = \h$, $\rho = \qfh$, and $\sigma = \eh$. 
 \end{proof}

\begin{corollary}\label{Cor:WOmega_dvr-h-descent}
Let $k$ be a perfect field of characteristic $p>0$. For each $n \geqslant 0$, the sheaves $W\Omega^n_{\QQ,\dvr}$ are $\h$-sheaves on $\Sch(k)$. 
 \end{corollary}
 \begin{proof}
We have already seen that $W\Omega^n_{\QQ}$ is an unramified \'etale sheaf by Proposition~\ref{prop:Unramified} and by Proposition~\ref{prop:qfhDescent} a $\qfh$-sheaf as well. Thus Proposition \ref{prop:dvrHDescent} applies.
\end{proof}

\subsection{Topological Torsion}\label{sec:Tor}

For a presheaf $\F$ on $\Sch(S)$, we denote as in \cite{HuberKebekusKelly2016} by $\tor\F(X)$  the sections of $\F(X)$ which vanish on a dense open subscheme and we call a presheaf $\F$ torsion free if $\tor \F(X) = 0$. An important observation of \cite[Sec.~5]{HuberKebekusKelly2016} is that understanding the torsion of a presheaf $\F$ on $\Sch(S)$ is crucial to link $\F_{\dvr}$ with various sheafifications. Therefore, we will look at the torsion forms for $W \Omega^n_{\QQ}$ over a perfect field of positive characteristic. 

\begin{lemma}\label{RemTorReg} 
An unramified presheaf $\F$ has no torsion forms on $\Reg(S)$. 
\end{lemma}
\begin{proof}
Let $\F$ be an unramified sheaf on $\Reg(S)$ and $X\in\Reg(S)$. Suppose $\omega \in \F$ is torsion. This means that there is an open dense subscheme $U\hookrightarrow X$ such that $\omega\big\vert_U=0$. However, the second property for unramified presheaves implies that the induced morphism $\F(X)\rightarrow \F(U)$ is injective. Thus $\omega=0$.
\end{proof}

This means in particular that the sheaves $W\Omega^n$ and $W\Omega^n_{\QQ}$ have no torsion forms on $\Reg(k)$.

\begin{corollary}\label{Cor:Torsion}
Let $\F$ be an unramified presheaf on $\Reg(S)$. If $\F_{\dvr}$ is an $\h$-sheaf, then $\F_{\dvr}$ is torsion free.
\end{corollary}

\begin{proof}
Since $\F_{\dvr}$ agrees with $\F$ on $\Reg(S)$ and $\F$ is torsion free on $\Reg(S)$ by the previous lemma, it is immediately clear that $\F_{\dvr}$ is torsion free on $\Reg(S)$. As  $\F_{\dvr}$ is $\h$-separated by assumption and alterations are $\h$-covers, it follows directly from de~Jong's alteration theorem that $\F_{\dvr}$ is torsion free on $\Sch(S)$. 

More precisely, let $X\in\Sch(S)$ and $\omega\in \F_{\dvr}(X)$ be a torsion form, meaning that there is an open dense subscheme $U\hookrightarrow X$ such that $\omega\big\vert_U=0$. One has  to show that $\omega=0$.  Because $\F_{\dvr}$ satisfies $\h$-descent, it suffices to show that there is an $\h$-cover on which $\omega$ vanishes. 

By de Jong's alteration \cite{deJong1996} theorem, there exists an alteration $\alpha:X'\rightarrow X$, such that $X'$ is regular. Let $U'$ be the preimage of $U$ under $\alpha$, which is a dense open in $X'$. The hypothesis $\omega\big\vert_U=0$ implies that $\alpha^\ast\omega \big\vert_{U'}=\omega_{X'} \big\vert_{U'}=0$. But as $X'$ is regular and $\F_{\dvr}$ torsion free on $\Reg(S)$ it follows, that $\alpha^\ast\omega=0$ and hence $\omega=0$.
\end{proof}

\begin{corollary}\label{Cor:torfree}
The sheaf $W\Omega^n_{\QQ,\dvr}$ is torsion free on $\Sch(k)$. 
\end{corollary}
\begin{proof}
As $W\Omega^n_{\QQ,\dvr}$ is an $\h$-sheaf by Corollary~\ref{Cor:WOmega_dvr-h-descent}, the previous corollary applies. 
\end{proof}

\section{Properties of $\h$-Witt differentials} 

We come now to the main result of this paper. Equipped with the findings from the previous section we can show that $W\Omega^n_{\QQ}$ satisfies $\h$-descent on regular schemes over a perfect field of positive characteristic. We also highlight several  consequences of this which are in line with some of the results in \cite{HuberJoerder2014} for $\Omega^n$ in characteristic $0$. 

\subsection{Regular $\h$-descent for Witt differentials}

As we have observed, $W\Omega^n$ and $W\Omega^n_{\QQ}$ as unramified presheaves are torsion free on regular schemes. The following general lemma and its corollary show how this property might help to understand the $\h$-sheafification of a Zariski sheaf. 

\begin{lemma}\label{lem:KillingTorsion} 
Let $\F$ be a Zariski sheaf on $\Sch(S)$ which is torsion free  on $\Reg(S)$. For $X \in \Sch(S)$ and $\omega \in \tor \F(X)$ there is an alteration $\alpha:  \widetilde{X} \to X$ such that $\alpha^{\ast}\omega$ vanishes in $\F(\widetilde{X})$. 
\end{lemma}
\begin{proof}
As $\omega$ is torsion, there is a Zariski open subset $U\subset X$ such that the restriction $\omega\big\vert_U=0$. If $\alpha:\widetilde{X}\rightarrow X$ is  a regular alteration, then $\alpha^\ast\omega\in \F(\widetilde{X})$ vanishes on $\alpha^{-1}U$, which is dense open in $\widetilde{X}$. But since $\widetilde{X}$ is regular, $\F(\widetilde{X})$ is torsion free by assumption and thus $\alpha^\ast\omega=0$. 
\end{proof}

\begin{corollary}\label{Cor:KillingTorsion}
Let $\F$ be a Zariski sheaf on $\Sch(S)$ which is torsion free on $\Reg(S)$. For $X\in\Sch(S)$, let $\omega\in\F(X)$ be an element which vanishes on the generic points of $X$. Then there is an $\h$-morphism $X'\rightarrow X$, such that $\omega\big\vert_{X'}=0$.
\end{corollary}
\begin{proof}
Without loss of generality we may assume that $X$ is integral by considering each irreducible component of $X$ separately and using the fact that each scheme is $\h$-locally reduced. In this case, the hypothesis on $\omega$ means that $\omega$ is torsion on $X$ and we can now apply Lemma~\ref{lem:KillingTorsion}.
\end{proof}

We are now ready to give the  descent theorem for Witt differentials.

\begin{theorem}\label{dRWSmHDescent}
Let $k$ be a perfect field of positive characteristic $p$. For any $X\in\Reg(k)$, the change of topology maps induce  isomorphisms $W\Omega^n_{\QQ}(X)\cong W\Omega^n_{\QQ,\h}(X)$ for all $n\geqslant 0$.
\end{theorem}

\begin{proof}
Since $W\Omega^n_{\QQ,\dvr}$ is an $\h$-sheaf by Corollary \ref{Cor:WOmega_dvr-h-descent}, we have by the universal property of sheafification a factorization
	\begin{equation}\label{MapsOfSheaves}
	W\Omega^n_{\QQ}\rightarrow W\Omega^n_{\QQ,\h}\rightarrow W\Omega^n_{\QQ,\dvr}
	\end{equation}
of the canonical morphism $W\Omega^n_{\QQ}\rightarrow W\Omega^n_{\QQ,\dvr}$. We show in two steps that this is actually an isomorphism on $\Reg(k)$.

The first step is to show, that the second map is a monomorphism. This follows by  a similar argument to \cite[Prop.~5.12]{HuberKebekusKelly2016}. Let $X\in\Sch(k)$ and let $\omega$ be a section in the kernel of $W\Omega^n_{\QQ,\h}(X)\rightarrow W\Omega^n_{\QQ,\dvr}(X)$. Choose an $\h$-cover $f:X'\rightarrow X$, such that $f^\ast\omega$ is in the image of $W\Omega^n_{\QQ}(X')\rightarrow W\Omega^n_{\QQ,\h}(X')$. This way we obtain an element $\omega'$ in the kernel of $W\Omega^n_{\QQ}(X')\rightarrow W\Omega^n_{\QQ,\dvr}(X')$, and have to show that it vanishes on some $\h$-cover of $X'$. By  Corollary~\ref{Cor:KillingTorsion}, it is enough to show that $\omega'$ is trivial on every generic point $x\in X'$. Such a point is isomorphic to the generic point of a regular scheme $V\in\Reg(X')$. Since $\omega'$ is trivial in $W\Omega^n_{\QQ,\dvr}(X')$ it vanishes by definition of the extension functor on every regular scheme over $X'$, specifically on $V$ as well as on its generic point and therefore on $x$.

For the second step, let $X\in\Reg(k)$. Apply the global section functor with $X$ to the the factorization (\ref{MapsOfSheaves}) to obtain
	$$
	W\Omega^n_{\QQ}(X)\rightarrow W\Omega^n_{\QQ,\h}(X)\hookrightarrow W\Omega^n_{\QQ,\dvr}(X).
	$$
By Remark~\ref{Rem:4.6} the canonical map $W\Omega^n_{\QQ}\rightarrow W\Omega^n_{\QQ,\dvr}$ is an isomorphism on regular schemes. Consequently, the above composition on $X$ is an isomorphism  where the second map is injective by the first step. This implies that the second map is surjective and therefore an isomorphism. It follows right away that the same is true for the first map.
\end{proof}

\begin{remark}
In the case when $n=0$ there are stronger $\h$-descent results known. Namely, the sheaf $W\OO_{\QQ}$ satisfies cohomological $\h$-descent on $\Sch(k)$ by \cite[Thm.~2.4]{BerthelotBlochEsnault2007}. This is even true for $W \OO$  after just inverting the Witt vector Frobenius \cite[Thm.~5.40]{BhattScholze2015}. This is unlike the situation for $\Omega^0=\OO$ in characteristic $0$, where one has such a result only for semi-normal schemes \cite[Prop~4.5]{HuberJoerder2014}.
\end{remark}

It is well known to experts that $\h$-sheaves on $k$ can be interpreted in terms of Zariski sheaves on smooth schemes. The reason is that by de~Jong's alteration theorem $k$-schemes are $\h$-locally smooth. More precisely we have the following statement.

\begin{proposition}
Let $k$ be a perfect field of characteristic $p>0$. For  a $k$-variety $X$ and an unramified Zariski sheaf $\F$ on $\Reg(k)$ which  satisfies regular $\h$-descent, we have an isomorphism
	$$ 
	\F_{\h}(X) \cong  \varprojlim_{Y \in \Reg(X)} \F(Y).
	$$
\end{proposition}
\begin{proof}
The proof follows almost  the same steps as \cite[Cor.~3.9]{HuberJoerder2014}. Explicitly, we may write the projective limit in the statement as
	$$
	 \varprojlim_{Y \in \Reg(X)} \F(Y) = \left\{(\alpha_{f}) \in\prod_{\substack{f \colon Y \to X \\ Y \text{regular}}} \F(Y)  \quad\bigg|\quad \forall \psi:\begin{aligned} \xymatrix{ Y' \ar[d]_\psi \ar[dr]^{f'} & \\ Y \ar[r]_f & X}\end{aligned} \Rightarrow \psi^\ast \alpha_{f}=\alpha_{f'} \right\}.
	$$
 By de Jong's alteration theorem \cite{deJong1996}, $X$ has a regular $\h$-cover. Hence there is a natural injection
	$$
	\F_{\h}(X)\hookrightarrow \varprojlim_{Y \in \Reg(X)} \F(Y)
	$$
because a section $\beta\in\F_{\h}(X)$ determines uniquely a compatible family $\beta_f:=f^\ast\beta\in\F_{\h}(Y)$ where $f:Y\rightarrow X$ with $Y$ regular varies. By  hypothesis, $\F$ has regular $\h$-descent, and hence $\beta_f\in \F_{\h}(Y) =  \F(Y)$ for all $Y \in \Reg(X)$. It remains to show that this natural injection is  surjective. 

We choose regular $\h$-covers $g:X'\rightarrow X$ and $h:X''\rightarrow X'\times_X X'$, and denote by $i:X'\times_X X'\rightarrow X$ the canonical map and by $\pr_i X'\times_X X'\rightarrow X'$, $i=1,2$, the projections, fitting into the diagram
	$$
	\xymatrix{X''\ar[dr]^h & & \\
	& X'\times_X X' \ar[r]^{\pr_2}\ar[d]^{\pr_1}\ar[dr]^i & X'\ar[d]^g\\
	& X'\ar[r]^g& X.}
	$$ 
For a compatible family $(\alpha_f)_{f}$, where $f$ runs over all $f \colon Y \to X$ and $Y \in \Reg(X)$, we thus have
$$(\pr_1\circ h)^\ast \alpha_g=\alpha_{i\circ h}=(\pr_2\circ h)^\ast \alpha_g.$$
Consequently, $\F_{\h}(X)$ has a unique section $\beta$ such that $\alpha_g=g^\ast\beta$. 

To show that $\beta$ is the desired preimage of $(\alpha_f)_{f}$, it suffices now to show the equality
	$$
	f^\ast(\beta)=\alpha_f
	$$
for any morphism $f:Y\rightarrow X$ with $Y$ regular. For this let $e:Y'\rightarrow Y\times_X X'$ be a regular $\h$-cover,  and denote by $j:Y\times_X X'\rightarrow X$ the canonical map, and by  $\pr_Y: Y\times_X X'\rightarrow Y$ and  $\pr_{X'}:Y\times_X X'\rightarrow X'$ respectively the projections, which again fit in a diagram
	$$
	\xymatrix{Y'\ar[dr]^e &&\\
	& Y\times_X X' \ar[r]^{\pr_{X'}} \ar[d]^{\pr_Y} \ar[dr]^j & X' \ar[d]^g\\
	& Y\ar[r]^f & X.}
	$$
We therefore have equalities
$$(\pr_Y\circ p)^\ast \alpha_f=\alpha_{j\circ p} = (\pr_{X'}\circ p)^\ast \alpha_g=(\pr_{X'}\circ p)^\ast g^\ast\beta=(\pr_Y\circ p)^\ast f^\ast\beta.$$
As $\F$ is unramified, the map $(\pr_Y\circ p)^\ast:\F(Y)\rightarrow \F(Y)$ is injective, and we can conclude that 
$\alpha_f=f^\ast\beta$. 
\end{proof}

\begin{corollary}\label{cor:h-decomp}
Let $k$ be a perfect field of characteristic $p>0$. For $X \in \Sch(k)$ and each $n\geqslant 0$, 
$$
W \Omega_{\QQ,\h}^n(X) \cong  \varprojlim_{Y \in \Reg(X)} W \Omega_{\QQ}^n(Y).
$$
\end{corollary}

\begin{remark}
Note that the limit appearing in the above proposition coincides in fact with the definition of the functor $(-)_{\dvr}$, so that one has in the situation of the corollary an identification
	$$
	W \Omega_{\QQ,\h}^n(X)\cong W \Omega_{\QQ,\dvr}^n(X).
	$$
\end{remark}

We end this section with a simple formula for the global sections of $W \Omega_{\QQ,\h}^n$ (c.f. \cite[Rem.~3.8]{HuberJoerder2014}).

\begin{proposition}\label{prop:regulardescentExact} 
Let $k$ be a perfect field of characteristic $p>0$. Then for $X \in \Sch(k)$, and regular $\h$-covers $X' \rightarrow X$ and  $X'' \rightarrow X' \times_X X'$, the sequence 
	\begin{equation}\label{SheafProperty4}
	0\rightarrow W \Omega_{\QQ,\h}^n(X) \rightarrow W \Omega_{\QQ}^n(X')\rightarrow W \Omega_{\QQ}^n(X'')
	\end{equation}
is exact for all $n\geqslant 0$.
\end{proposition}

\begin{proof} 
For the $\h$-cover $X' \rightarrow X$ we obtain by the sheaf condition an exact sequence
	$$
	0\rightarrow W \Omega_{\QQ,\h}^n(X)\rightarrow W \Omega_{\QQ,\h}^n(X')\rightarrow W \Omega_{\QQ,\h}^n(X'\times_X X').
	$$
 For the $\h$-cover $X''\rightarrow X'\times_X X'$, the induced map $W \Omega_{\QQ,\h}^n(X'\times_X X')\rightarrow W \Omega_{\QQ,\h}^n(X'')$ is injective, too,  by the same reasoning. Together, this gives an exact sequence
	$$
	0\rightarrow W \Omega_{\QQ,\h}^n(X)\rightarrow W \Omega_{\QQ,\h}^n(X')\rightarrow W \Omega_{\QQ,\h}^n(X'').
	$$
As $X'$ and $X''$ are regular,  $W \Omega_{\QQ,\h}^n(X')\cong W \Omega_{\QQ}^n(X')$ and $W \Omega_{\QQ,\h}^n(X'')\cong W \Omega_{\QQ}^n(X'')$ because of regular $\h$-descent and we obtain the desired exact sequence.
\end{proof}

\subsection{Basic properties and some examples}

Throughout this section fix a perfect field  $k$  of characteristic $p>0$. We investigate now some properties of $\h$-Witt differentials in the spirit of \cite[Prop.~4.2]{HuberJoerder2014}.  For $X \in \Sch(k)$, we denote by  $W \Omega_{\QQ,\h}^n|_X$  the push-forward of $W \Omega_{\QQ,\h}^n$ to the Zariski-site of $X$. Furthermore, if $j: X^{\text{reg}}\rightarrow X$ is the inclusion of the regular locus, we use the notation $W\Omega_{\QQ,X}^{[n]} := j_\ast W\Omega^n_{\QQ,X^{\text{reg}}}$.

\begin{proposition}\label{prop:BasicProperties}
For $X\in\Sch(k)$, the sheaf $W \Omega_{\QQ,\h}^n|_X$ satisfies the following properties.
	\begin{enumerate}
	\item The sheaf $W\Omega_{h,\QQ}^n|_X$ is a quasi-coherent torsion free sheaf of $W\OO_{\QQ,X}$-modules, 
	\item If $X$ is regular, then 
		$$
		W\Omega^n_{\QQ,\h}\big|_X\cong W\Omega^n_{\QQ,X}.
		$$
	\item Let $r:X_{\text{red}}\rightarrow X$ be the reduction. Then 
		$$
		r_\ast W\Omega^n_{\QQ,\h}\big|_{X_{\text{red}}} \cong W\Omega^n_{\QQ,\h}\big|_X.
		$$
	\item If $X$ is reduced, there is a natural injection
		$$
		W\Omega_{\QQ,X}^n/\text{torsion}\hookrightarrow W\Omega_{\QQ,\h}^n\big|_X.
		$$
	\item Let $X$ be normal and $j:X^{\text{reg}}\rightarrow X$ the inclusion of the regular locus. There is a natural injection 
		$$
		W\Omega_{\QQ,\h}^n\hookrightarrow W\Omega_{\QQ,X}^{[n]}.
		$$
	\item For $n>\dim(X)$, $W\Omega_{\QQ,\h}^n\big|_X=0$.
	\end{enumerate}
\end{proposition}

\begin{proof}
To show (i),  we choose as before smooth $\h$-covers $g:X'\rightarrow X$ and $h:X''\rightarrow X'\times_X X'$, denote by $i:X'\times_X X'\rightarrow X$ the canonical map and set $\widetilde{h}=i\circ h$. Without loss of generality, we may assume that $g$ and $\widetilde{h}$ are proper. Hence they are quasi-compact and quasi-separated. It follows that the direct image sheaves $g_\ast W\Omega^n_{\QQ}$ and $\widetilde{h}_\ast W\Omega^n_{\QQ}$ are quasi-coherent, because this is the case for $W\Omega^n_{\QQ}$. According to the exact sequence (\ref{SheafProperty4}) $W\Omega_{h,\QQ}^n|_X$ is the kernel of a morphism $g_\ast W\Omega^n_{\QQ}\rightarrow \widetilde{h}_\ast W\Omega^n_{\QQ}$ of quasi-coherent sheaves, and thus  itself quasi-coherent. Torsion freeness  for $W\Omega^n_{\QQ,\dvr}=W\Omega^n_{\QQ,\h}$ was already shown in  Corollary~\ref{Cor:torfree}.  

Item (ii) follows immediately from regular $\h$-descent in Theorem~\ref{dRWSmHDescent}. Item (iii) is true as the inclusion of the reduced subscheme is an $\h$-morphism. 

For (iv) consider the natural map $	W\Omega^n_{\QQ,X}\rightarrow W\Omega^n_{\QQ,\h}\big|_X$. Since $W\Omega_{\QQ,\h}^n\big|_X$ is torsion free by (i), we may quotient out torsion in the domain to obtain 
	$$
	W\Omega^n_{\QQ,X}/\text{torsion}\rightarrow W\Omega^n_{\QQ,\h}\big|_X.
	$$
Let $j:X^{\text{reg}}\hookrightarrow X$ be the regular locus. By the definition of torsion freeness we have 
	$$
	W\Omega_{\QQ,\h}^n\big|_X\hookrightarrow  j_\ast W\Omega_{\QQ,\h}^n\big|_{X^{\text{reg}}}\cong j_\ast W\Omega^n_{\QQ,X^{\text{reg}}},
	$$
where the last isomorphism comes from (ii). Note that $j_\ast W\Omega^n_{\QQ,{X^{\text{reg}}}}$ is torsion free and that it agrees with $W\Omega^n_{\QQ,X}/\text{torsion}$ on ${X^{\text{reg}}}$. Hence
	$$
	W\Omega^n_{\QQ,X}/\text{torsion} \hookrightarrow j_\ast W\Omega^n_{\QQ,{X^{\text{reg}}}}
	$$
and the diagram $$\xymatrix{W\Omega_{\QQ,\h}^n\big|_X \ar@{^{(}->}[r] & j_\ast W\Omega^n_{\QQ,{X^{\text{reg}}}}\\ W\Omega^n_{\QQ,X}/\text{torsion} \ar[u] \ar@{^{(}->}[ru]&}$$ commutes, so that the vertical map is injective as well. 

Now we consider a normal scheme $X$  and let again $j:{X^{\text{reg}}}\hookrightarrow X$ be its regular locus. Because $W\Omega_{\QQ,\h}^n\big|_X$ is torsion free, the restriction $W\Omega_{\QQ,\h}^n(X)\hookrightarrow W\Omega_{\QQ,\h}^n(X^{\text{reg}}) \cong W\Omega_{\QQ}^n(X^{\text{reg}})$ is injective. As this is a local property (v) follows. 

For (vi), assume without loss of generality that $X$ is reduced, and by torsion freeness restrict to an open subset where it is smooth. The vanishing of $W\Omega^n_{\QQ,\h}\big|_X$ for $n>\dim(X)$ then follows from the vanishing of $W\Omega^n_{\QQ}$ for smooth schemes. 
\end{proof}

To conclude we record two corollaries  for mildly singular cases. 

\begin{corollary}
For $X\in\Sch(k)$ with normal crossings, there is an isomorphism
	$$
	W\Omega^n_{\QQ,\h}\big|_X \cong W\Omega^n_{\QQ,X}\slash{\text{torsion}}.
	$$
\end{corollary}
\begin{proof}
The proof is a direct translation of \cite[Prop.~4.9]{HuberJoerder2014} with their \cite[Prop.~4.2]{HuberJoerder2014} replaced by Proposition~\ref{prop:BasicProperties}. We repeat the argument for convenience of the reader. 

Since $X$ is reduced there is by Proposition~\ref{prop:BasicProperties}(iv) an injection 
	$$
	W\Omega_{\QQ,X}^n/\text{torsion}\hookrightarrow W\Omega_{\QQ,\h}^n\big|_X.
	$$
Working \'etale locally, one may assume that $X$ is a union of smooth hyperplane sections and use induction with respect to the number of irreducible components $c(X)$ of $X$. 

If $c(X)\leqslant 1$ then we are in the smooth case, and the statement is just Proposition~\ref{prop:BasicProperties}(ii). Thus let $c(X)>1$. Choose an irreducible component $Z\subset X$, let $X'=\overline{X\backslash Z}$ and let $E$ be the inverse image of $Z$ under the map $X'\rightarrow X$. Then $(X',Z)$ is an abstract blow-up of $X$ and by  \cite[Prop.~3.3]{HuberJoerder2014}  there is an exact sequence 
	$$
	0\rightarrow W\Omega^n_{\QQ,\h}\big|_X \rightarrow W\Omega^n_{\QQ,\h}\big|_{X'}\oplus W\Omega^n_{\QQ,\h}\big|_Z \rightarrow W\Omega^n_{\QQ,\h}\big|_E.
	$$
As in the classical case, we see that the pull-back of torsion free Witt differentials is again torsion free. 
Hence, the above sequence fits into a commutative diagram
	$$
	\xymatrix{ 0\ar[r] & W\Omega^n_{\QQ,\h}\big|_X \ar[r] & W\Omega^n_{\QQ,\h}\big|_{X'}\oplus W\Omega^n_{\QQ,\h}\big|_Z \ar[r]& W\Omega^n_{\QQ,\h}\big|_E\\
	& W\Omega^n_{\QQ,X}\slash{\text{torsion}} \ar[r] \ar@{^{(}->}[u] & W\Omega^n_{\QQ,X'}\slash{\text{torsion}}\oplus W\Omega^n_{\QQ,Z}\slash{\text{torsion}} \ar[r] \ar[u]^{\sim} & 		W\Omega^n_{\QQ,E}\slash{\text{torsion}} \ar@{^{(}->}[u]}
	$$
where the first and last vertical map are inclusions by (iv) of Proposition~\ref{prop:BasicProperties} and the middle vertical map is an isomorphism by (ii) and by induction respectively. 

By local calculations the second line in the diagram is also exact. Thus a diagram chase shows that the first vertical map is an isomorphism as well. 
\end{proof}

\begin{corollary}
Let $X\in\Reg(k)$ be quasi-projective with an action of a finite group $G$. Then one has 
	$$
	W\Omega^n_{\QQ,\h}\big|_{X/G}\cong (W\Omega^n_{\QQ,X})^G.
	$$
\end{corollary}

\begin{proof}
As $X/G$ is normal and the projection $X\rightarrow X/G$ is ramified with Galois group $G$, we have by Galois descent
	$$
	W\Omega^n_{\QQ,\h}\big|_{X/G}\cong \left( W\Omega^n_{\QQ,\h}\big|_{X} \right)^G
	$$
and by Proposition~\ref{prop:BasicProperties}(ii) $W\Omega^n_{\QQ,\h}\big|_{X}\cong W\Omega^n_{\QQ,X}$. 	
\end{proof}

\begin{remark}
Considering that $X/G$ is normal there is according to Proposition~\ref{prop:BasicProperties}(v) an inclusion 
	$$
	W\Omega^n_{\QQ,\h}\big|_{X/G}\subset W\Omega^{[n]}_{\QQ,X/G}.
	$$ 
In analogy to the classical case, one expects this to be an isomorphism.
\end{remark}


\begin{thebibliography}{99}

\bibitem[AHL16]{AnconaHuberPepinLehalleur2016} \textsc{G.~Ancona, A.~Huber and S.~Pepin~Lehalleur}: \textit{On the relative motive of a commutative group scheme}. Algebraic Geometry 3 (2), (2016). 

\bibitem[Bei12]{Beilinson2012} \textsc{A.~Beilinson}: \textit{$p$-adic periods and derived de~Rham cohomology}. J. of the AMS 25, no. 3, 715--732, (2012).

\bibitem[BBE07]{BerthelotBlochEsnault2007} \textsc{P.~Berhtelot, S.~Bloch and H.~Esnault}: \textit{On Witt vector cohomology for singular varieties.} Compos. Math., 143(2), 363--392, (2007). 

\bibitem[BS17]{BhattScholze2015} \textsc{B.~Bhatt and P.~Scholze}: \textit{Projectivity of the Witt vector affine Grassmannian}. Invent. Math., 209(2): 329--423,  (2017). 

\bibitem[CD12]{CisinskiDeglise2012} \textsc{D.-C.~Cisinksi and F.~ D\'eglise}: \textit{Triangulated categories of mixed motives}. Available at \url{https://arxiv.org/abs/0912.2110}, (2012).

\bibitem[CHSW05]{CortinasHasemeyerSchlichtingWeibel2005} \textsc{G.~Cortinas, C.~H\"asemeyer, M.~Schlichting and C.~A.~Weibel }  \textit{Cyclic homology, cdh-cohomology and negative $K$-theory}, Ann. of Math., 167, 549--573 , (2008).

\bibitem[dJ96]{deJong1996} \textsc{A.~J.~de Jong}: \textit{Smoothness, semi-stability and alterations}. Publ. Math. IH{\'E}S 83: 51--93, (1996).

\bibitem[EGA]{EGA4} \textsc{A.~Grothendieck and J.~Dieudonn\'e}. \textit{\'El\'ements de g\'eom\'etire alg\'ebrique: IV. \'Etude locale des sch\'emas et des
morphismes de sch\'emas, Quatri\`eme partie}. Publ. Math. IHES, \textbf{32}, 5--361, (1967).

\bibitem[GK15]{GabberKelly2015} \textsc{O.~Gabber and S.~Kelly}: \textit{Points in algebraic geometry} J. Pure Appl. Algebra, 219, 10, 4667--4680, (2015). 

\bibitem[Gei06]{Geisser2006} \textsc{T.~Geisser}: \textit{Arithmetic cohomology over finite field and special values of $\zeta$-functions.} Duke Math. J., 133(1): 27--57, 05, (2006). 

\bibitem[GH09]{GeisserHesselholt2009} \textsc{T.~Geisser and L.~Hesselholt}: \textit{On the vanishing of negative $K$-groups}. Math. Ann.,  348: 707--736, (2010). 

\bibitem[God58]{Godement1958} \textsc{R.~Godement}: \textit{Topologie alg\'ebrique et th\'eorie des faisceaux}.  Hermann, Paris, (1958).

\bibitem[Gre76]{Greco1976} \textsc{S.~Greco}: \textit{Two theorems on excellent rings}, Nagoya Math. J., 60, (1976), 139--149.

\bibitem[Gro85]{Gros1985} \textsc{M.~Gros}: \textit{Classes de {Chern} et classes de cycles en cohomologie de {de~Rham--Witt} logarithmique}. M\'em. Soc. Math. France 2\textsuperscript{{e}} s\'erie, \textbf{21}: 1--87, (1985).

\bibitem[HJ14]{HuberJoerder2014} \textsc{ A.~Huber and C.~J\"order }: \textit{ Differential forms in the h-topology}, Algebr. Geom. 1, no. 4: 449--478, (2014). 

\bibitem[HK17]{HuberKelly2017} \textsc{A.~Huber and S.~Kelly}: \textit{ Differential forms in postiive characteristic II: cdh-descent via funtorial Riemann-Zariski spaces}. Available at \url{https://arxiv.org/abs/1706.05244}, (2017).

\bibitem[HKK16]{HuberKebekusKelly2016} \textsc{A.~Huber, S.~Kebekus and S.~Kelly}: \textit{Differential forms in positive characteristic avoiding resolution of sngularities}. Bull. Soc. Math. Fr., to appear, available at \url{https://arxiv.org/abs/1407.5786}, (2016). 

\bibitem[Ill79]{Illusie1979} \textsc{L.~Illusie}: \textit{Complexe de de~Rham--Witt et cohomologie cristalline}. Ann. Sci. de l'E.N.S. 4\textsuperscript{e} s\'erie, tome 12, n\textsuperscript{o} 4: 501--661, (1979).

\bibitem[Kel12]{Kelly2012} \textsc{S.~Kelly}: \textit{Triangulated categories of motives in positive characteristic}. Ph.D. Thesis, Universit\'e Paris 13, Australian National University, available at \url{http://arxiv.org/abs/1305.5349}, (2012). 

\bibitem[Koh17]{Kohrita2017} \textsc{T.~Kohrita}: \textit{Deligne-Beilineson cycle maps for Lichtenbaum cohomology}. Available at \url{https://arxiv.org/abs/1703.09493v1}, (2017).

\bibitem[Kol97]{Kollar1997} \textsc{J.~Koll\'ar}: \textit{Quotient spaces modulo algebraic groups}. Ann. of math. (2), 145(1): 33--79, (1997). 

\bibitem[Lur09]{Lurie2009} \textsc{J.~Lurie}: \textit{Higher Topos Theory}. Annals of Mathematics Studies 170, Princton University Press, (2009). 

\bibitem[Mor12]{Morel2012} \textsc{F.~Morel}: \textit{$\AA^1$-algebraic topology over a field}. Lecture Notes in Mathematics \textbf{2052}, Springer, Heidelberg, (2012). 

\bibitem[SGA1]{SGA1} \textit{Rev\`etements \'etales et groupe fondamental}. Documents math\'ematiques (Paris), 3. Soci\'et\'e Math\'ematique de France, Paris, (2003). S\'eminaire de G\'eom\'etrie alg\'ebrique du Bois Maris 1960--61,directed by A.~Grothendieck, with two papers by M.~Raynaud, updated and annotatedreprint of the 1971 original [Lecture Notes in Math. 224, Springer, Berlin; MR0354651 (50$\#$7129)].

\bibitem[SV96]{SuslinVoevodsky1996} \textsc{A.~Suslin and V.~Voevodsky}: \textit{Singular homology of abstract algebraic varieties}. Invent. Math., 123(1): 61--94, (1996).

\bibitem[Voe96]{Voevodsky1996} \textsc{V.~Voevodsky}: \textit{Homology of schemes}. Selecta Math. (N.S.), 2(1): 111--153, (1996).

\bibitem[Voe10]{Voevodsky2000} \textsc{V.~Voevodsky}: \textit{Unstable motivic homotopy categories in Nisnevich and $\cdh$-topologies}.  J. Pure Appl. Algebra, 214, (2010), 1399--1406.

\end{thebibliography}
\end{document}